%% file: main.tex
\documentclass[reqno]{amsart}

\input{Preamble}

\title{Efficiently distinguishing all tangles \linebreak in locally finite graphs}
\author{Raphael W.\ Jacobs \and Paul Knappe}
\address{Universität Hamburg, Department of Mathematics, Bundesstraße 55 (Geomatikum), 20146 Hamburg, Germany}
\email{\{raphael.jacobs, paul.knappe\}@uni-hamburg.de}

\@namedef{subjclassname@2020}{\textup{2020} Mathematics Subject Classification}

\begin{document}

\begin{abstract}
    While finite graphs have tree-decompositions that efficiently distinguish all their tangles, \hbox{locally} finite graphs with thick ends need not have such tree-decompositions.
    We show that every locally finite graph without thick ends admits such a tree-decomposition, in fact a canonical one. 
    Our proof exhibits a thick end at any obstruction to the existence of such tree-decompositions and builds on new methods for the analysis of the limit behaviour of strictly increasing sequences of separations.
\end{abstract}

\maketitle

\section{Introduction}

In their Graph Minors Project~\cite{GM}, Robertson and Seymour introduced the notion of tangles as a novel way to capture highly cohesive substructures, or `clusters', of a graph.
Tangles capture such clusters indirectly in that they encode the clusters' positions in the graph:
a tangle is an orientation of all the low-order separations of the graph, orienting every such separation towards the side where the cluster is located.
Robertson and Seymour's innovation was to view such `consistent orientations' themselves as abstract clusters,  thus providing a general framework for many concrete cluster types.

One of their key results about tangles is the \emph{tree-of-tangles~theorem}, which asserts that the tangles in a finite graph can be arranged in a tree-like way~\cite{GMX}.
Diestel, Hundertmark and Lemanczyk found such a tree-like arrangement which is even \emph{canonical}, that is, the automorphisms of the graph act naturally on it:

\begin{theorem}[\cite{ProfilesNew}*{Theorem 3}] \label{thm:finiteTDToT}
    Every connected finite graph has a canonical tree-decomposition which efficiently distinguishes all its tangles.
\end{theorem}

\noindent Here, we say that a separation of~$G$ \emph{distinguishes} two tangles in~$G$ if they orient that separation differently, and it does so \emph{efficiently} if it has minimum order among all the separations of~$G$ that distinguish these two tangles.
If two tangles are distinguished by some separation of~$G$, they are called \emph{distinguishable}.
Now a tree-decomposition~$(T, \cV)$ of a graph~$G$ \emph{efficiently distinguishes} all the tangles in~$G$ if, for every two distinguishable tangles in~$G$, there exists a separation of~$G$ which is induced by an edge of the decomposition tree~$T$ and efficiently distinguishes these two tangles.

The set of separations induced by a tree-decomposition 
is itself tree-like~\cites{TreelikeSpaces,TreeSets} in that the separations are pairwise nested.
In finite graphs, these two tree-like structures are equivalent, since the converse holds as well: every nested set of separations of a finite graph induces a tree-decompo\-sition~\cite{confing}*{Theorem~4.8} that gives rise to it.
This converse, however, fails for locally finite graphs, potentially infinite graphs in which every vertex has only finitely many neighbours.

The resulting difference between nested sets of separations and tree-decompositions of locally finite graphs also appears in the context of tree-like arrangements of tangles.
Elbracht, Kneip and Teegen showed on the one hand that~\cref{thm:finiteTDToT} does not directly extend to such graphs:
they constructed a locally finite graph that does not have a tree-decomposition that efficiently distinguishes all its tangles, canonical or not~\cite{InfiniteSplinters}*{Example~4.9}\footnote{
Carmesin, Hamann and Miraftab independently constructed such an example in their arXiv version \cite{CanonicalTreesofTDs}*{Example~7.4} but it is not present in their journal version.}.
On the other hand, they and, independently, Carmesin, Hamann and Miraftab proved that every locally finite graph still admits a \emph{tree of tangles}, a nested set~$N$ of separations such that every two distinguishable tangles are efficiently distinguished by some element of~$N$ and every separation in~$N$ efficiently distinguishes some two tangles: \looseness=-1

\begin{theorem}[\cite{CanonicalTreesofTDs}*{Theorem 7.3} \& \cite{InfiniteSplinters}*{Theorem 6.6}] \label{thm:ToT}
    Every connected locally finite graph has a canonical tree of tangles.
\end{theorem}

So in order to understand to which locally finite graphs~\cref{thm:finiteTDToT} extends, i.e., which locally finite graphs admit tree-decompositions that efficiently distinguish all their tangles, we have to investigate what prevents a tree of tangles, which exists by~\cref{thm:ToT}, from inducing a tree-decomposition.
All known examples of locally finite graphs whose trees of tangles do not induce tree-decompositions share one feature:
such graphs have a \emph{thick} end, that is, an end containing infinitely many disjoint rays.
In this paper, we show that thick ends must appear in all such examples:

\begin{mainresult} \label{main:Theorem}
    Every connected locally finite graph without thick ends has a canonical tree-decomposition which efficiently distinguishes all its tangles.
\end{mainresult}

\noindent Note that the converse of~\cref{main:Theorem} does not hold, as the infinite grid witnesses.
More generally, \cite{GraphDec}*{Lemma~8.5} implies that every connected locally finite graph that is quasi-transitive and accessible admits a tree-decomposition as in~\cref{main:Theorem} -- no matter if it has a thick end or not. \\

For a proof of~\cref{main:Theorem}, it suffices to show that some (canonical) tree of tangles, which exists by~\cref{thm:ToT}, induces a (canonical) tree-decomposition if the locally finite graph has no thick end. 
We prove that, in fact, \emph{every} tree of tangles of such a graph does so. 

In order to determine whether a nested set~$N$ of separations, such as a tree of tangles, induces a tree-decomposition or not, we study the limit behaviour of sequences of oriented separations in~$N$ that are strictly increasing with respect to the natural partial order of oriented separations:
Let the \emph{limit} of a strictly increasing sequence~$((A_i, B_i))_{i \in \N}$ of separations be~$(A ,B) := (\bigcup_{i \in \N} A_i, \bigcap_{i \in \N} B_i)$, which is again a separation of the graph.
Elbracht, Kneip and Teegen \cite{InfiniteSplinters}*{Lemma~2.7} showed that if the limit separation~$(A, B)$ of every strictly increasing sequence in a (canonical) nested set~$N$ of separations is trivial, in that~$B = \emptyset$, then $N$ induces a (canonical) tree-decomposition.

Our strategy for the proof~\cref{main:Theorem} now is as follows.
Let~$G$ be a locally finite graph, and let~$N$ be a tree of tangles of~$G$.
If~$N$ does not induce a tree-decomposition, then there exists a strictly increasing sequence~$((A_i, B_i))_{i \in \N}$ of separations in~$N$ whose limit~$(A, B)$ satisfies~$B \neq \emptyset$.
We show that the separator $A \cap B$ of this limit separation must be infinite.
Then some end of~$G$ is in the closure of~$A \cap B$ 
in the Freudenthal compactification of $G$.
It will be easy to see that this end~$\omega$ is unique.

The main effort of the proof will be to show that this end~$\omega$ is thick.
To do so, we carefully analyse the sequence~$((A_i, B_i))_{i \in \N}$ and the interaction with its limit~$(A, B)$.
In particular, we exploit the fact that~$((A_i, B_i))_{i \in \N}$ consists of separations in the tree of tangles~$N$, that is, each~$\{A_i, B_i\}$ efficiently distinguishes some two tangles in~$G$.
Our \emph{interlacing theorem}, \cref{lem:ProfileInterlacing}, obtains from~$((A_i, B_i))_{i \in \N}$ a strictly increasing sequence~$((A_i', B_i'))_{i \in \N}$ of separations in~$N$ with the same limit~$(A, B)$ and an accompanying sequence of tangles in~$G$.
This sequence of tangles not only witnesses that each~$\{A_i', B_i'\}$ efficiently distinguishes some two tangles in~$G$, but is also `interlaced' with~$((A_i', B_i'))_{i \in \N}$.
With this structure at hand, we then deduce that the unique end~$\omega$ 
in the closure of~$A \cap B$ must be thick.

All in all, our proof method finds a thick end at the limit of every obstruction~$((A_i, B_i))_{i \in \N}$ that prevents a tree of tangles from inducing a tree-decomposition:

\begin{mainresult} \label{main:Technical}
    Let~$N$ be a tree of tangles of a connected locally finite graph~$G$, and let~$((A_i, B_i))_{i \in \N}$ be a strictly increasing sequence of separations in~$N$ with limit~$(A, B)$.
    If~$B \neq \emptyset$, then~$G$ has a unique end that is 
    in the closure of~$A \cap B$, and this end is thick.
\end{mainresult}

Our paper is structured as follows.
We first recall some relevant definitions and the corresponding notation in~\cref{sec:Prelimns}.
In~\cref{sec:LimitSeparations} we study limit separations and develop various general techniques for their investigation.
Then we introduce our interlacing theorem in \cref{sec:InterlacingMethod}. Finally, we
combine our results on limit separations with the interlacing theorem to prove our main results, \cref{main:Theorem,main:Technical}, in~\cref{sec:ProofMainTheorem}.

\section{Preliminaries} \label{sec:Prelimns}

For the general definitions and notations around graphs, we follow~\cite{DiestelBook16}.
In this section we recall the relevant definitions around separations of graphs and introduce related notation.

A \emph{separation} of a graph~$G$ is an unordered pair~$\{A, B\}$ of subsets of~$V(G)$ such that~$A \cup B = V(G)$ and there is no edge in~$G$ with one endvertex in~$A \setminus B$ and the other one in~$B \setminus A$.
The \emph{order} $|A, B|$ of a separation~$\{A, B\}$ is the cardinality of its \emph{separator}~$A \cap B$.

The two \emph{orientations} of a separation~$\{A, B\}$ are the ordered pairs~$(A, B)$ and~$(B, A)$ which are \emph{oriented separations} of~$G$ and whose \emph{underlying} separation is~$\{A, B\}$.
Given a set~$S$ of separations, we write~$\vS$ for the set~$\{(A, B), (B, A) \mid \{A, B\} \in S\}$ of all orientations of separations in~$S$.
An \emph{orientation} of a set~$S$ of separations of~$G$ is a subset of~$\vS$ containing precisely one orientation of each separation in~$S$.

We equip the set of oriented separations of a graph~$G$ with their natural partial order:
\begin{equation*}
    (A, B) \le (C, D) : \iff A \subseteq C \text{ and } B \supseteq D.
\end{equation*}
Two separations are \emph{nested} if they have 
orientations which are comparable with respect to~$\le$, and a set of separations is \emph{nested} if its elements are pairwise nested.
We further say that two oriented separations are \emph{nested} if their underlying separations are nested.
If two separations are not nested, then they \emph{cross}.
We will use the following observation which can easily be derived from the definition of the partial order of oriented separations.

\begin{lemma}[\cite{confing}*{Statement~(6)}] \label{lem:NestedWithCorners}
    For two oriented separations~$(A, B)$ and~$(C, D)$ of a graph~$G$, we have~$(A, B) \le (C, D)$ if and only if~$(A \cap D) \setminus S = \emptyset$ where~$S := (A \cap B) \cap (C \cap D)$.
    In particular, two separations~$\{A, B\}$ and~$\{C, D\}$ of~$G$ cross if and only if all four sets~$A \cap C$, $A \cap D$, $B \cap C$ and~$B \cap D$ are non-empty after removing~$S$.
\end{lemma}

A separation~$\{A, B\}$ of a graph~$G$ is~\emph{proper} if~$A, B \neq V(G)$.
Note that if~$(A, B)$ satisfies~$(A, B) \le (C, D)$ and~$(A, B) \le (D, C)$ for some separation~$\{C, D\}$, then~$B = V(G)$ and hence~$\{A, B\}$ is not proper.
It is immediate from the definition of separation that the separator~$A \cap B$ of a proper separation~$\{A,B\}$ of~$G$ is non-empty if~$G$ is connected.

For better readability, we will often use the term `separation' for both separations and oriented separations, if the meaning is clear from the context.
Analogously, we use terms such as `order' which are defined for separations also for oriented separations if there is no ambiguity in how the definitions transfer.

We finally remark that whenever we speak of (sub)sequences, then we assume them to be infinite.

\section{Limit separations} \label{sec:LimitSeparations}

In this section, we study sequences of (oriented) separations and their limits.
We begin by defining limit separations and then note some direct consequences of the definition.
Next, we investigate the limit separations of strictly increasing sequences of separations which are `non-exhaustive'.
Finally, we describe how the relation of a finite-order separation to such a limit separation is connected to its relation with the separations in the sequence.

\subsection{Definition and basic properties}

Let~$((A_i, B_i))_{i \in \N}$ be a sequence of separations of a graph~$G$.
Its~\emph{supremum}~$(A, B)$ (with respect to the partial order~$\le$ of oriented separations) is given by~$A := \bigcup_{i \in \N} A_i$ and~$B := \bigcap_{i \in \N} B_i$.
We remark that $(A_i, B_i) \le (A, B)$ for all $i \in \N$ by definition.
The next lemma shows that~$(A, B)$ is again a separation of~$G$. We will use it throughout this paper tacitly.

\begin{lemma} \label{lemma:limitsepissep}
    Let~$((A_i, B_i))_{i \in \N}$ be a sequence of separations of a graph~$G$.
    Then its supremum~$(A, B)$ is a separation of~$G$.
\end{lemma}
\begin{proof}
    Suppose for a contradiction that there exists an edge~$e \in G$ with endvertices~$a \in A \setminus B$ and~$b \in B \setminus A$.
    It is immediate from the definition of~$A$ and~$B$ that~$b \in B_i \setminus A_i$ for all~$i \in \N$.
    The definition of~$B$ also implies that there exists~$I \in \N$ with~$a \notin B_I$, and hence~$a \in A_I \setminus B_I$.
    So~$e$ is an edge of~$G$ joining~$A_I \setminus B_I$ and~$B_I \setminus A_I$, which contradicts that~$(A_I, B_I)$ is a separation of~$G$.
\end{proof}

If a sequence of separations `dominates' another sequence, then this transfers to their suprema.
To make this precise, a sequence~$((A_i, B_i))_{i \in \N}$ of separations of a graph~$G$ \emph{dominates} a sequence~$((C_j, D_j))_{j \in \N}$ of separations of $G$ if for every~$j \in \N$ there exists~$i = i(j) \in \N$ such that~$(C_j, D_j) \le (A_i, B_i)$.
If two sequences of separations 
dominate each other, then they are \emph{interlaced}.
Note that an increasing sequence of separations is interlaced with each of its subsequences.

The next lemma is an immediate consequence of the above definitions. 
It describes how the suprema of two sequences of separations are related to each other if one sequence dominates the other or the sequences are interlaced.
\begin{lemma} \label{lem:InterlacingLemma}
    Let~$((A_i, B_i))_{i \in \N}$ and~$((C_j, D_j))_{j \in \N}$ be two sequences of separations of a graph, and let~$(A, B)$ and~$(C, D)$ be their suprema, respectively.
    If~$((A_i, B_i))_{i \in \N}$ dominates~$((C_j, D_j))_{j \in \N}$, then
    $(C, D) \le (A, B)$.
    In particular, if~$((A_i, B_i))_{i \in \N}$ and~$((C_j, D_j))_{j \in \N}$ are interlaced, then~$(A, B) = (C, D)$. \qed
\end{lemma}

While the above two lemmas deal with general sequences of separations, we are mainly interested in sequences which are increasing or even strictly increasing with respect to the partial order of oriented separations.
For an increasing sequence~$((A_i, B_i))_{i \in \N}$ of separations of a graph~$G$, we refer to its supremum~$(A, B)$ as its~\emph{limit}, as this separation is indeed the limit of the sequence with respect to pointwise convergence. \looseness=-1

It is immediate from this definition of the limit separation~$(A, B)$ that, for every finite set~$X \subseteq A$, there exists~$I \in \N$ such that~$X \subseteq A_i$ for all~$i \ge I$, and that, for every finite set~$X \subseteq A \setminus B = V(G) \setminus B$, there exists~$I \in \N$ such that~$X \subseteq V(G) \setminus B_i = A_i \setminus B_i$ for all~$i \ge I$.
Combining these two observations directly yields the following interplay of finite sets with the limit separation and the corresponding sequence.

\begin{lemma} \label{lem:PushingLemma}
    Let~$((A_i, B_i))_{i \in \N}$ be an increasing sequence of separations of a graph, and let~$(A, B)$ be its limit.
    For every finite set~$X \subseteq A$, there exists~$I \in \N$ such that, for all~$i \ge I$, the intersections of~$X$ with~$A \cap B$ and~$A \setminus B$ equal the respective intersections with~$A_i \cap B_i$ and~$A_i \setminus B_i$. \qed
\end{lemma}

\subsection{Non-exhaustive sequences and their limit separations}

Our main motivation for the study of limit separations of strictly increasing sequences comes from the next lemma by Elbracht, Kneip and Teegen~\cite{InfiniteSplinters}*{Lemma~2.7} which establishes a condition for when a nested set of separations induces a tree-decomposition.
We will later use this lemma in the proof of~\cref{main:Theorem} to obtain a tree-decomposition as in~\cref{main:Theorem} from the nested set given by~\cref{thm:ToT}.

To state this lemma, let us call a strictly increasing sequence~$((A_i, B_i))_{i \in \N}$ of separations of a graph~$G$ \emph{exhaustive} if its limit~$(A, B)$ satisfies~$B = \emptyset$~\cite{GraphDec}*{Section~7.1}.
Furthermore, a nested set~$N$ of separations of~$G$ is~\emph{($\omega$-)exhaustive} if every strictly increasing sequence in~$\vN$ is exhaustive.

\begin{lemma}[\cite{InfiniteSplinters}*{Lemma 2.7}] \label{lem:ExhaustiveImpliesTreeDecomp}
    Let~$N$ be a nested set of proper finite-order
    separations of a connected graph~$G$.
    If~$N$ is exhaustive, then it induces a tree-decomposition~$(T, \cV)$ of~$G$.
    Moreover, if~$N$ is canonical, then~$(T, \cV)$ is canonical as well.
\end{lemma}

In order to prove that a nested set~$N$ of separations is exhaustive, it 
suffices to consider certain `representative' sequences in~$\vN$ with strong structural properties instead of just arbitrary ones.
\cref{lem:InterlacingLemma} provides us a tool for this:
when we investigate the limit of a strictly increasing sequence in~$\vN$, then~\cref{lem:InterlacingLemma} allows us to instead consider any interlaced strictly increasing sequence as they have the same limit. 
As an important example of this, we will demonstrate such a use of~\cref{lem:InterlacingLemma} in combination with the next~\cref{lem:BoundedOrder} in the paragraph after stating the latter.

To state~\cref{lem:BoundedOrder}, recall that, given a set~$X$ of vertices of a graph~$G$, a component~$K$ of~$G - X$ is \emph{tight} if~$N_G(K) = X$, and a separation~$\{A, B\}$ of~$G$ is \emph{tight} if~$A \setminus B$ and~$B \setminus A$ each contain the vertex set of a tight component of~$G - (A \cap B)$~(cf.~\cite{InfiniteSplinters}*{Section~6}).
Note that tight separations are proper.
We remark that the property `tight' of a component $K$ of $G-X$ heavily depends on the knowledge of the deleted set $X$ via context.

\begin{lemma}[\cite{GraphDec}*{Lemma 7.5}] \label{lem:BoundedOrder}
    Let~$((A_i, B_i))_{i \in \N}$ be a strictly increasing sequence of tight separations of a connected locally finite graph~$G$.
    If there exists~$K \in \N$ such that~$|A_i,B_i| \le K$ for all~$i \in \N$, then~$((A_i, B_i))_{i \in \N}$ is exhaustive.
\end{lemma}

We use \cref{lem:InterlacingLemma} to show that
if a nested set~$N$ of separations consists of tight separations, then in order to determine whether~$\vN$ is exhaustive, it suffices to investigate strictly increasing sequences in~$\vN$ whose orders are strictly increasing.
Note that by applying \cref{lem:InterlacingLemma} again on a suitable subsequence we see that it suffices that the orders of the elements of every non-exhaustive strictly increasing sequence $(A_i,B_i)_{i \in \N}$ in~$\vN$ \emph{cofinitely exceed every integer}, that is for every $k \in \N$ there exists $I \in \N$ such that $|A_i, B_i| \geq k$ for every~$i \geq I$.
Indeed, suppose that for some $k \in \N$ there exists no such $I \in \N$. Then there exists a subsequence of $(A_i, B_i)_{i \in \N}$ whose orders are bounded from above by $k$.
This subsequence is exhaustive by \cref{lem:BoundedOrder} which contradicts by \cref{lem:InterlacingLemma} that $(A_i, B_i)_{i \in \N}$ is not exhaustive.

In the remainder of this section, we show that if a strictly increasing sequence of tight separations is non-exhaustive, then its limit separation has infinite order.
\begin{theorem} \label{lem:LimitSepHasInfiniteOrder}
    Let~$((A_i, B_i))_{i \in \N}$ be a strictly increasing sequence of tight separations of a connected locally finite graph~$G$, and let~$(A, B)$ be its limit.
    If~$((A_i, B_i))_{i \in \N}$ is non-exhaustive, then~$A \cap B$ is infinite. \looseness=-1
\end{theorem}

\noindent Towards the proof of~\cref{lem:LimitSepHasInfiniteOrder}, we show two auxiliary results.
The first lemma observes that~$A \cap B$ is non-empty even under weaker assumptions.
\begin{lemma} \label{lem:LimitSeparatorIsNonEmpty}
    Let~$((A_i, B_i))_{i \in \N}$ be a strictly increasing sequence of separations of a connected graph, and let~$(A, B)$ be its limit.
    If~$((A_i, B_i))_{i \in \N}$ is non-exhaustive, then~$A \cap B$ is non-empty.
\end{lemma}
\begin{proof}
    Since the sequence~$((A_i, B_i))_{i \in \N}$ is strictly increasing, some~$A_i$, and thus~$A$, is non-empty.
    The set~$B$ is non-empty, as~$((A_i, B_i))_{i \in \N}$ is non-exhaustive.
    So since~$G$ is connected and~$(A, B)$ is a separation of~$G$, $A \cap B$ has to be non-empty as well.
\end{proof}

The second lemma shows that if the separations in a sequence are tight, then its supremum is again close to being tight.
\begin{lemma} \label{lem:limitsepispseudotight}
    Let~$((A_i, B_i))_{i \in \N}$ be an increasing sequence of tight separations of a locally finite graph~$G$, and let~$(A, B)$ be its supremum.
    If~$B$ is non-empty, then every~$v \in A \cap B$ has a neighbour~$w \in B \setminus A$ such that, for infinitely many~$i \in \N$, $w$ is contained in a tight component~$K_i$ of~$G - (A_i \cap B_i)$ with~$ V( K_i ) \subseteq B_i \setminus A_i$.
    In particular, we have~$N_G(B\setminus A) = A \cap B$.
\end{lemma}
\begin{proof}
    Consider an arbitrary vertex~$v \in A \cap B$.
    For every~$i \in \N$, the separation~$\{A_i, B_i\}$ is tight by assumption and 
    hence there exists a tight component~$K_i$ of~$G-(A_i \cap B_i)$ with~$V( K_i ) \subseteq B_i \setminus A_i$.
    Since~$G$ is locally finite, $v$ has only finitely many neighbours. 
    As $v \in A \cap B$, all but at most finitely many $K_i$ contain $v$ in their neighbourhood.
    Thus, the pigeonhole principle yields a neighbour~$w$ of~$v$ which is contained in infinitely many~$K_i$; let us denote the collection of all such~$i$ by~$\cI$.
    Then we have $w \in \bigcap_{i \in \cI} V( K_i ) \subseteq \bigcap_{i \in \cI} ( B_i \setminus A_i ) = B \setminus A$, where the last equality follows from~\cref{lem:InterlacingLemma} since~$((A_i,B_i))_{i \in \cI}$ is a subsequence of increasing sequence~$((A_i,B_i))_{i \in \N}$.
    The `in particular'-part follows immediately.
\end{proof}

\noindent We briefly remark that the above proof of~\cref{lem:limitsepispseudotight} did not use the fact that there also exists a tight component of~$G - (A_i \cap B_i)$ in~$ G[ A_i \setminus B_i ]$ 
for all~$i \in \N$.
Indeed, it is enough to assume that the separation~$\{A_i, B_i\}$ is proper and that there is a tight component in~$ G[ B_i \setminus A_i ]$.

With~\cref{lem:LimitSeparatorIsNonEmpty} and~\cref{lem:limitsepispseudotight} at hand, we are ready to prove~\cref{lem:LimitSepHasInfiniteOrder}.

\begin{proof}[Proof of~\cref{lem:LimitSepHasInfiniteOrder}]
    Suppose for a contradiction that~$A \cap B$ is finite.
    Then~\cref{lem:PushingLemma} yields~$I_1 \in \N$ such that~$A \cap B \subseteq A_i \cap B_i$ for all~$i \ge I_1$.
    Moreover, since~$((A_i, B_i))_{i \in \N}$ is non-exhaustive, \cref{lem:InterlacingLemma} and~\cref{lem:BoundedOrder} together imply that there is an integer~$I_2 \ge I_1$ such that~$|A_i, B_i| > |A, B|$ for all~$i \ge I_2$; in particular, we have~$A \cap B \subsetneq A_i \cap B_i$ for all~$i \ge I_2$.
    Furthermore, $A \cap B \neq \emptyset$ by~\cref{lem:LimitSeparatorIsNonEmpty}, and~\cref{lem:limitsepispseudotight} thus yields a neighbour~$w \in B\setminus A$ of~$A \cap B$ and an infinite set~$\cI \subseteq \N$ such that~$w$ is contained in a tight component~$K_i$ of~$G - (A_i \cap B_i)$ with~$V( K_i ) \subseteq B_i \setminus A_i$ for all~$i \in \cI$.
    
    Fix any integer~$i \in \cI$ with~$i \ge I_2$.
    Since~$i \ge I_2$, there exists some vertex~$u \in (A_i \cap B_i) \setminus (A \cap B)$.
    So~$u \in A_i \cap B_i = N_G(K_i)$, since~$K_i$ is a tight component of~$G - (A_i \cap B_i)$, and~$w \in K_i$.
    Thus, there exists a $u$--$w$~path~$P$ in~$G$ such that all vertices of~$P$ except~$u$ are contained in~$K_i$.
    We claim that~$P$ contradicts that~$(A, B)$ is a separation of~$G$.

    To see this, recall that~$w \in B \setminus A$.
    We also have~$u \in A \setminus B$, as~$u \in A_i \setminus (A \cap B)$ by its choice and~$A_i \subseteq A$.
    Moreover, $K_i$ is a component of~$G - (A_i \cap B_i)$ and hence disjoint from~$A_i \cap B_i$ and in particular from~$A \cap B$, since~$A \cap B \subseteq A_i \cap B_i$ for~$i \ge I_2 \ge I_1$.
    Altogether, the path~$P$ joins~$u \in A \setminus B$ and~$w \in B \setminus A$, but avoids~$A \cap B$ which contradicts that~$(A, B)$ is a separation of~$G$.
\end{proof}

By~\cref{lem:LimitSepHasInfiniteOrder}, the separator~$A \cap B$ of the limit separation~$(A, B)$ of a non-exhaustive strictly increasing sequence of separations is infinite.
As we will see in~\cref{sec:ProofMainTheorem}, this is already the first step towards~\cref{main:Technical} which asserts that the graph has a thick end which is 
in the closure of~$A \cap B$ in the Freudenthal compactification of $G$.
Indeed, for every infinite set of vertices in a connected locally finite graph, such as~$A \cap B$, there exists an end which is 
in the closure of this set.
So it remains to show that the end arising in this way is thick which will be the main challenge in~\cref{sec:ProofMainTheorem}. \looseness=-1

\subsection{Interplay with the partial order of oriented separations}

Consider the limit separation~$(A, B)$ of an increasing sequence~$((A_i, B_i))_{i \in \N}$ of separations, and a further fixed separation~$\{C, D\}$.
Given the relation of~$(A, B)$ and the orientations of~$\{C, D\}$ with respect to the partial order of oriented separations, what can we say about the relation of the $(A_i, B_i)$ and the orientations of~$\{C, D\}$?
This question is of particular relevance for the analysis of limit separations that we will perform in~\cref{sec:ProofMainTheorem}, and we shall address this question here.

We start with the case that no orientation of a separation~$\{C, D\}$ is comparable with the limit separation~$(A, B)$, i.e.\ that~$\{C, D\}$ crosses~$\{A, B\}$.

\begin{lemma} \label{lem:FiniteOrderCrossLimit}
    Let~$((A_i, B_i))_{i \in \N}$ be an increasing sequence of separations of a graph~$G$, and let~$(A, B)$ be its limit.
    If a separation~$\{C, D\}$ of~$G$ crosses~$\{A, B\}$, then there exists~$I \in \N$ such that~$\{C, D\}$ crosses~$\{A_i, B_i\}$ for all~$i \ge I$.
\end{lemma}

\begin{proof}
    We prove the contrapositive of this statement.
    Assume that a separation $\{ C,D\}$ of $G$ is nested with infinitely many $\{A_i,B_i\}$.
    Denote the infinite set of these indices $i$ by $I$.
    First, assume that there is some~$j \in I$ such that $(C,D) \leq (A_j,B_j)$ or $(D,C) \leq (A_j, B_j)$.
    Since $((A_i, B_i))_{i \in \N}$ is increasing, $(C,D) \leq (A_i, B_i)$ for every $i \in \N$ with $i \geq j$.
    Thus, $(C,D) \leq (A,B)$ or $(D,C) \leq (A,B)$, respectively.
    Otherwise, we have~$(A_j,B_j) \leq (C,D)$ or $(A_j,B_j) \leq (D,C)$ for every $j \in I$.
    Since $I$ is infinite, we may assume by symmetry that there are infinitely many $j \in I$ with $(A_j,B_j) \leq (C,D)$.
    Thus, $(A,B) \leq (C,D)$.
\end{proof}

We now turn to the case that some orientation of a separation~$\{C, D\}$ is comparable with the limit separation~$(A, B)$, i.e.\ that~$\{C, D\}$ is nested with~$\{A, B\}$.
It is immediate from the definition of limit separation that if an orientation of~$\{C, D\}$, say~$(C, D)$, is greater than~$(A, B)$, then~$(C, D)$ is also greater than all the~$(A_i, B_i)$ in the sequence defining~$(A, B)$.

It remains to consider the case that an orientation of a separation~$\{C, D\}$, say~$(C, D)$, is smaller than the limit separation~$(A, B)$.
For this case, we restrict our attention to finite-order separations~$(C, D)$, as infinite-order separations do not allow any conclusion on their relation with the~$(A_i, B_i)$ in the sequence defining~$(A, B)$.
The next two lemmas describe the relation of~$(C, D)$ with the~$(A_i, B_i)$, depending on whether the sequence~$((A_i, B_i))_{i \in \N}$ is exhaustive or not.

First, let us assume that~$((A_i, B_i))_{i \in \N}$ is exhaustive.
Then both orientations of~$\{C, D\}$ are smaller than~$(A, B)$.
\cref{lem:ExhaustiveFiniteOrderBelowLimit} asserts that if all the~$B_i$ are connected, then some orientation of~$\{C, D\}$ is also smaller than~$(A_i, B_i)$ for all but finitely many~$i \in \N$.

\begin{lemma}\label{lem:ExhaustiveFiniteOrderBelowLimit}
    Let~$((A_i,B_i))_{i \in \N}$ be a strictly increasing sequence of separations of a graph~$G$ such that~$G[B_i]$ is connected for every~$i \in \N$, and let~$\{C,D\}$ be a finite-order separation of~$G$.
    If~$((A_i,B_i))_{i \in \N}$ is exhaustive, then there is an~$I \in \N$ such that either~$(C,D) \le (A_i,B_i)$ for all~$i \ge I$ or~$(D,C) \le (A_i,B_i)$ for all~$i \ge I$.
\end{lemma}
\begin{proof}
    Since~$((A_i, B_i))_{i \in \N}$ is exhaustive, we have~$B = \emptyset$ and hence~$C \cap D \subseteq A = A \setminus B$.
    By assumption, $C \cap D$ is finite, so we can apply~\cref{lem:PushingLemma} to~$C \cap D$ and obtain an~$I \in \N$ with~$C \cap D \subseteq A_I \setminus B_I$.
    Since~$G[B_I]$ is connected by assumption and~$C \cap D$ is disjoint from~$B_I$, this implies that~$B_I$ is contained either in~$C \setminus D$ or in~$D \setminus C$.
    If~$B_I \subseteq C \setminus D$, then~$A_I \supseteq V(G) \setminus B_I \supseteq D$ and hence~$(D, C) \le (A_I, B_I)$; analogously, we have~$(C, D) \le (A_I, B_I)$ if~$B_I \subseteq D \setminus C$.
    The claim then follows directly from the fact that the sequence~$((A_i, B_i))_{i \in \N}$ is increasing.
\end{proof}

Secondly, we consider the case that~$((A_i, B_i))_{i \in \N}$ is non-exhaustive.
For sequences of tight separations in connected locally finite graphs, \cref{lem:FiniteOrderBelowLimit} then asserts that~$(C, D)$ is smaller than all but finitely many~$(A_i, B_i)$.

\begin{lemma} \label{lem:FiniteOrderBelowLimit}
    Let~$((A_i, B_i))_{i \in \N}$ be a strictly increasing sequence of tight separations of a connected locally finite graph~$G$, let~$(A, B)$ be its limit, and let~$(C, D)$ be any finite-order separation of~$G$.
    If~$((A_i, B_i))_{i \in \N}$ is non-exhaustive and~$(C, D) \le (A, B)$, then there exists an~$I \in \N$ such that
    $(C, D) \le (A_I, B_I)$.
\end{lemma}
\begin{proof}
    The sequence~$((A_i, B_i))_{i \in \N}$ is increasing, so it suffices to find a single~$I \in \N$ with~$(C, D) \le (A_I, B_I)$.
    To do so, let~$S_i := (A_i \cap B_i) \cap (C \cap D)$ for every~$i \in \N$.
    We show the existence of some~$I \in \N$ with \hbox{$(C \cap B_I) \setminus S_I = \emptyset$}, which then yields the claim by~\cref{lem:NestedWithCorners}.
    More precisely, we successively find conditions on~$i \in \N$ for the emptiness of each of the three sets~$X_i := ((C \cap B_i) \setminus S_i) \cap D$, $Y_i := ((C \cap B_i) \setminus S_i) \cap A_i$ and~$Z_i := ((C \cap B_i) \setminus S_i) \setminus (A_i \cup D)$ (see \cref{fig:cornerdiagram}), which partition~$(C \cap B_i) \setminus S_i$, and finally give some~$I \in \N$ satisfying all these conditions.

    \begin{figure}[ht]
        \centering
        \includegraphics{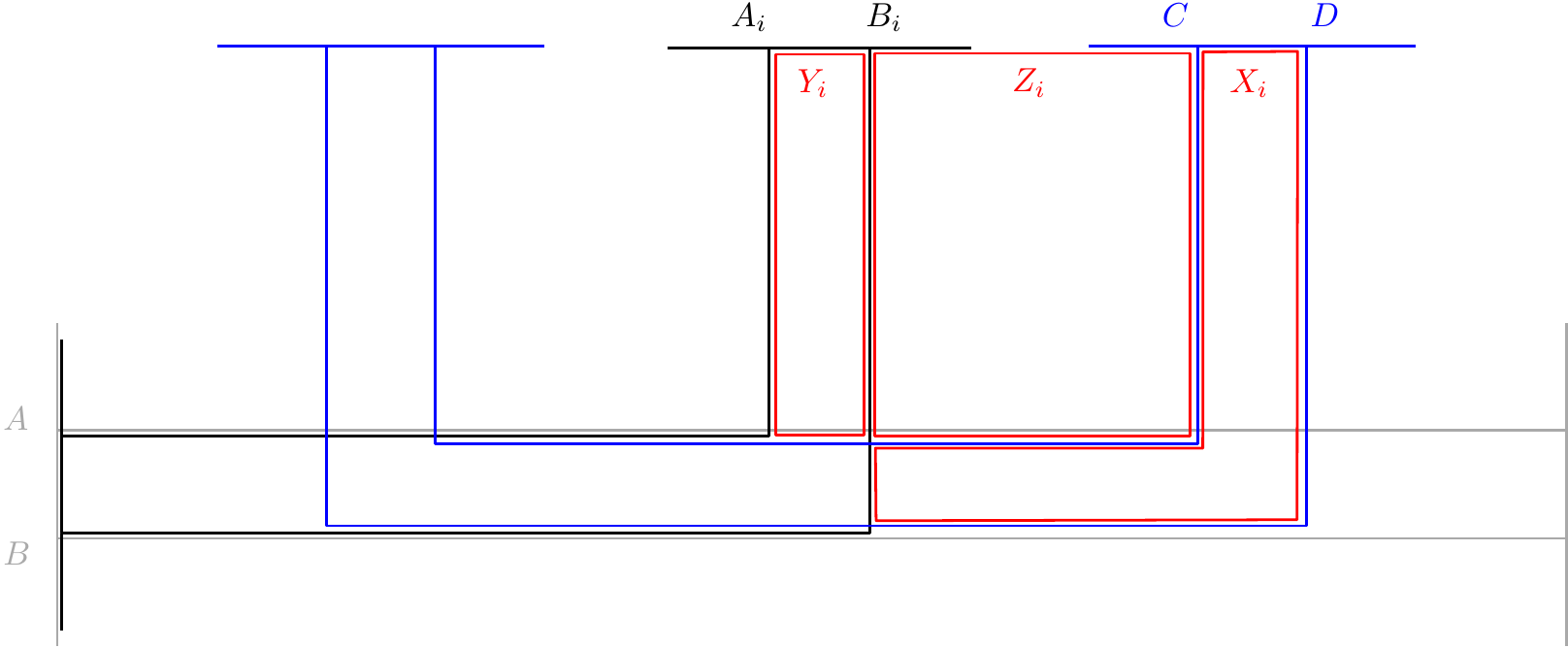}
        \caption{A diagram depicting the three sets $X_i$, $Y_i$ and $Z_i$ from the proof of \cref{lem:FiniteOrderBelowLimit}.}
        \label{fig:cornerdiagram}
    \end{figure}

    First, we consider the set~$X_i = ((C \cap B_i) \setminus S_i) \cap D$.
    Since~$(C, D) \le (A, B)$, we have~$C \cap D \subseteq A$.
    So since~$C \cap D$ is finite, \cref{lem:PushingLemma} yields~$I_1 \in \N$ such that~$C \cap D \subseteq A_i$ for all~$i \ge I_1$.
    This immediately implies $X_i = (C \cap D) \cap (B_i \setminus A_i) = \emptyset$ for all~$i \ge I_1$.
    
    Next, we turn to~$Y_i = ((C \cap B_i) \setminus S_i)  \cap A_i$.
    Recall that~$A \cap B \neq \emptyset$ by~\cref{lem:LimitSeparatorIsNonEmpty}.
    So~\cref{lem:limitsepispseudotight} yields a neighbour~$w \in B \setminus A$ of~$A \cap B$ and an infinite set~$\cI \subseteq \N$ such that~$w$ is contained in a tight component~$K_i$ of~$G - (A_i \cap B_i)$ with~$V( K_i ) \subseteq B_i \setminus A_i$ for all~$i \in \cI$.
    Now let~$i \in \cI$ with~$i \ge I_1$, and suppose for a contradiction that~$
    Y_i$ is non-empty.
    Since~$N_G(K_i) = A_i \cap B_i$, there thus exists a path~$P$ from~$w$ to the set~$Y_i 
    \subseteq A_i \cap B_i$ whose vertices lie in~$K_i$ except its last vertex~$u$.
    Note that~$u \in C \setminus D$, as~$u \in Y_i = (C \cap B_i \cap A_i) \setminus D$, and~$w \in D \setminus C$, as~$w \in B \setminus A$ and~$(C, D) \le (A, B)$.
    Moreover, the path~$P$ does not meet~$C \cap D$ as~$(C \cap D) \cap K_i \subseteq (C \cap D) \cap (B_i \setminus A_i) = X_i$ is empty for~$i \ge I_1$, as shown above.
    Altogether, $P$ starts in~$w \in D \setminus C$ and ends in~$u \in C \setminus D$, but~$P$ avoids~$C \cap D$, which contradicts that~$(C, D)$ is a separation of~$G$.

    It remains to consider~$Z_i = ((C \cap B_i) \setminus S_i) \setminus (A_i \cup D) = (C \cap B_i) \setminus (A_i \cup D)$.
    We first prove that if $Z_i$ is non-empty for~$i \in \cI$ with~$i \ge I_1$, then it must contain a neighbour of~$S_i$.
    Secondly, we show that~$Z_i$ cannot contain a neighbour of~$S_i$ for large enough~$i$, which yields the desired emptiness of~$Z_i$ for large enough~$i \in \cI$.
    
    First, let us prove that if~$Z_i$ is non-empty for~$i \in \cI$ with~$i \ge I_1$, then it contains a neighbour of~$S_i$.
    It is easy to check that $(A_i \cup D, C \cap B_i)$ is again a separation of~$G$ with separator~$(C \cap B_i) \cap (A_i \cup D)$.
    As shown above, the two sets~$X_i = ((C \cap B_i) \setminus S_i) \cap A_i$ and~$Y_i = ((C \cap B_i) \setminus S_i) \cap D$ are empty since~$i \in \cI$ with~$i \ge I_1$.
    Thus, the separator of $(A_i \cup D, C \cap B_i)$ is in fact~$S_i$.
    Both~$Z_i = (C \cap B_i) \setminus (A_i \cup D)$ and~$(A_i \cup D) \setminus (C \cap B_i)$ are non-empty, where the latter follows from~$(A_i, B_i)$ being tight and hence~$A_i \setminus B_i$ being non-empty.
    So since~$(A_i \cup D, C \cap B_i)$ is a separation of the connected graph~$G$, its separator~$S_i$ has to be non-empty and~$Z_i$ contains a neighbour of~$S_i$.
    
    Secondly, let us show that no neighbour of $S_i$ is contained in~$Z_i$ for large enough~$i$.
    By applying~\cref{lem:PushingLemma} to the set~$C \cap D$, which is finite since~$(C, D)$ has finite order, we obtain~$I_2 \in \N$ such that~$S_i$ equals $S := (A \cap B) \cap (C \cap D)$ for all~$i \ge I_2$.
    Thus, it suffices to argue that no neighbour of~$S$ is contained in~$Z_i$ for large~$i$.
    Consider any neighbour~$u$ of~$S$ which is contained in~$C$, and hence in~$A$, since~$(C, D) \le (A, B)$.
    Then there exists~$I_u \in \N$ such that~$u \in A_i$ for all~$i \ge I_u$, and hence~$u \notin B_i \setminus A_i \supseteq (C \cap B_i) \setminus (A_i \cup D) = Z_i$ for~$i \ge I_u$.
    Let~$I_3$ be the maximum over all~$I_u$, and note that this maximum exists since the finite set~$S$ has only finitely many neighbours in the locally finite graph~$G$.
    Then~$Z_i$ contains no neighbour of~$S$ for all~$i \ge I_3$.
    
    Choosing~$I \in \cI$ with~$I \ge \max \{I_1, I_2, I_3\}$ thus completes the proof.
\end{proof}

\section{The interlacing theorem} \label{sec:InterlacingMethod}

Interlaced strictly increasing sequences of separations have the same limit, as~\cref{lem:InterlacingLemma} shows.
So when we investigate the limit separation of a strictly increasing sequence of separations, then we may instead consider it as the limit of an interlaced sequence with stronger structural properties.
One example for this approach is the combination of~\cref{lem:InterlacingLemma} and~\cref{lem:BoundedOrder} which shows that, given a non-exhaustive strictly increasing sequence, there exists an interlaced sequence, and even a subsequence of the given one, whose separations have strictly increasing order.
We remark that this technique was already crucial to the proof of~\cite{GraphDec}*{Lemma~7.1}.

In this section we introduce the \emph{interlacing theorem} which enables us to use the above described method of interlaced sequences for strictly increasing sequences of `tangle-distinguishing' separations.
Before we are ready to state the interlacing theorem, we introduce the notion of pre-tangles; with all the wording below, we imitate the special case of tangles~(cf.~\cite{DiestelBook16}*{Section~12.5}).
A~set~$O$ of oriented separations of a graph~$G$ is \emph{consistent} if we do not have~$(B, A) \le (C, D)$ for any two~$(A, B), (C, D) \in O$ with~$\{A, B\} \neq \{C, D\}$.
A~set~$P$ of oriented separations of~$G$ is a~\emph{pre-tangle in~$G$} if there exists~$k \in \N \cup \{\aleph_0\}$ such that~$P$ is a consistent orientation of the set of all the separations of~$G$ of order less than~$k$.
It is immediate from this definition that if a pre-tangle~$P$ in~$G$ orients some separation, then it also orients every separation of at most the same order.  

A separation of~$G$ \emph{distinguishes} two pre-tangles~$P$ and~$P'$ in~$G$ if both~$P$ and~$P'$ contain one of its orientations but not the same, and the separation does so~\emph{efficiently} if it has minimum order among all the separations distinguishing~$P$ and~$P'$.
Note that every separation which distinguishes two pre-tangles has finite order by the definition of pre-tangle.
Given a set~$\cP$ of pre-tangles in~$G$, a set~$N$ of separations of~$G$ \emph{(efficiently) distinguishes}~$\cP$ if every two pre-tangles in~$\cP$ that are distinguished by some separation of~$G$  are (efficiently) distinguished by some separation in~$N$.
Moreover, a separation of~$G$ is \emph{$\cP$-relevant} if it efficiently distinguishes some two pre-tangles in~$\cP$.

The interlacing theorem roughly speaking asserts that, given a set~$\cP$ of pre-tangles in a graph~$G$ and a nested set~$N$ of~$\cP$-relevant separations which distinguishes~$\cP$ efficiently, every non-exhaustive strictly increasing sequence in~$\vN$ admits an interlaced sequence in~$\vN$ whose~$\cP$-relevance is witnessed in a very strong sense by a corresponding sequence in~$\cP$.
More precisely, the interlacing theorem reads as follows.
\begin{theorem}[Interlacing Theorem] \label{lem:ProfileInterlacing}
    Let~$\cP$ be a set of pre-tangles in a connected locally finite graph~$G$, let~$N$ be a nested set of~$\cP$-relevant proper separations of~$G$ which efficiently distinguishes~$\cP$, and let $((A_i, B_i))_{i \in \N}$ be a strictly increasing sequence in~$\vN$.
    If~$((A_i, B_i))_{i \in \N}$ is non-exhaustive, then there exists a strictly increasing sequence~$((A_i', B_i'))_{i \in \N}$ in~$\vN$ that is interlaced with~$((A_i, B_i))_{i \in \N}$ and a sequence~$(P_i)_{i \in \N}$ in~$\cP$ such that
    \begin{enumerate}[label=(IM\arabic*), leftmargin=1.5cm]
        \item $(B_i', A_i') \in P_i$ and~$(A_i', B_i') \in P_{i+1}$ for all~$i \in \N$, \label{item:Orientation}
        \item every minimal-order separation among~$\{A_i', B_i'\}, \dots, \{A_{j-1}', B_{j-1}'\}$ efficiently distinguishes~$P_i$ \linebreak and~$P_j$ for all~$i, j \in \N$ with~$i < j$. \label{item:DistinguishingSequence}
    \end{enumerate}
    Moreover, the~$\{A_i', B_i'\}$ can be chosen such that~$(|A_i', B_i'|)_{i \in \N}$ is strictly increasing, and then~$\{A_i', B_i'\}$ \linebreak efficiently distinguishes~$P_i$ from every~$P_j$ with $j > i$.
\end{theorem}

\noindent The remainder of this section is devoted to the proof of~\cref{lem:ProfileInterlacing}.

For better readability, we shall in the following denote separations by lower-case letters such as~$s$ and its orientations as~$\vs$ and~$\sv$.
Once we denote one orientation of~$s$ as~$\vs$, we write~$\sv$ for the other one, and vice-versa.
Analogously, we denote oriented separations as lower-case letters with a forward or backward arrow on top such as~$\vs$ and~$\sv$, and we write~$s$ for their underlying separation.
The order of a separation~$s$ is denoted by~$|s|$.

The proof of~\cref{lem:ProfileInterlacing} is divided into the four~\cref{lem:InterlacingDistinguishingSequence,lem:InterlacingStronglyRelevant,lem:InterlacingStrictlyIncreasingOrder,lem:InterlacingThinOutForStrictlyIncreasing}; each of these lemmas may be of independent interest for the analysis of strictly increasing sequences.
As one ingredient for their proofs, we have the following easy observation about those separations in a nested set that distinguish two fixed pre-tangles. \looseness=-1
\begin{lemma} \label{obs:DistinguisherFormChain}
    Let~$N$ be a nested set of separations of a graph~$G$.
    Given two pre-tangles~$P$ and~$Q$ in~$G$, the set of all the oriented separations~$\vs \in P \cap \vN$ with~$\sv \in Q$ forms a chain.
\end{lemma}
\begin{proof}
    In order to show that this set is a chain, we have to prove that for any two~$\vr, \vs \in P \cap \vN$, we have either~$\vr < \vs$ or~$\vs < \vr$.
    Since~$r$ and~$s$ are contained in~$N$ and hence nested, one of $\vr < \vs$, $\vs < \vr$, $\rv < \vs$ and $\vs < \rv$ holds.
    The first and the second case are as desired, so it remains to show that the third and fourth one do not occur.
    We cannot have~$\rv < \vs$, as both~$\vr$ and~$\vs$ are contained in~$P$ which is consistent.
    For~$\vs < \rv$, note that this implies~$\vr < \sv$ which cannot be the case by the consistency of~$Q$, in which both~$\sv$ and~$\rv$ are contained.
\end{proof}

The first~\cref{lem:InterlacingDistinguishingSequence} shows that for the proof of \cref{lem:ProfileInterlacing}, we can replace~\cref{item:DistinguishingSequence} with a weaker property~\cref{item:DistinguishingPairwise} in terms of consecutive pre-tangles.
\begin{lemma} \label{lem:InterlacingDistinguishingSequence}
    Let~$(\vs_i)_{i \in \N}$ be a strictly increasing sequence of separations of a graph~$G$, and let~$(P_i)_{i \in \N}$ be a sequence of pre-tangles in~$G$.
    If~$(\vs_i)_{i \in \N}$ and~$(P_i)_{i \in \N}$ satisfy~\cref{item:Orientation} and~\cref{item:DistinguishingPairwise}, where
    \begin{enumerate}[label=(IM\arabic*'), leftmargin=1.5cm]
        \setcounter{enumi}{1}
        \item the separation~$s_i$ efficiently distinguishes~$P_i$ and~$P_{i+1}$ for all~$i \in \N$,
        \label{item:DistinguishingPairwise}
    \end{enumerate}
    then they also satisfy~\cref{item:DistinguishingSequence}.
\end{lemma}
\begin{proof}
    To show~\cref{item:DistinguishingSequence}, consider any two integers~$i, j \in \N$ with~$i < j$, and let~$\ell \in \{i, \dots, j-1\}$ such that~$s_\ell$ has minimal order among~$s_i, \dots, s_{j-1}$.
    Then both pre-tangles~$P_i$ and~$P_j$ in~$G$ orient~$s_\ell$, since~$P_i$ orients~$s_i$ by~\cref{item:Orientation} and~$|s_\ell| \le |s_i|$ by the choice of~$\ell$ and analogously for~$P_j$ and~$s_{j-1}$.
    Since~$\sv_i \in P_i$ and~$\vs_{j-1} \in P_j$ by~\cref{item:Orientation}, the consistency of~$P_i$ and~$P_j$ implies that~$\sv_\ell \in P_i$ and~$\vs_\ell \in P_j$, so~$s_\ell$ distinguishes~$P_i$ and~$P_j$.
    
    It remains to show that~$s_\ell$ distinguishes~$P_i$ and~$P_j$ even efficiently.
    Let~$r$ be an arbitrary separation of~$G$ with~$|r| < |s_\ell|$.
    By the choice of~$\ell$, we have~$|r| < |s_\ell| \le |s_k|$ for all~$k \in \{i, \dots, j-1\}$.
    So 
    for every~$k \in \{i, \dots, j-1\}$, both~$P_k$ and $P_{k+1}$ orient~$r$, and they both contain the same orientation of~$r$ by~\cref{item:DistinguishingPairwise}.
    Repeatedly applying this fact from~$k=i$ up to~$k=j-1$ implies that~$P_i$ and~$P_j$ contain the same orientation of~$r$.
    Thus, $r$ does not distinguish~$P_i$ and~$P_j$.
    So no separation of~$G$ of order less than~$|s_\ell|$ distinguishes~$P_i$ and~$P_j$; in particular, $s_\ell$ distinguishes~$P_i$ and~$P_j$ efficiently.
\end{proof}

The second~\cref{lem:InterlacingStronglyRelevant} asserts, in the setting of~\cref{lem:ProfileInterlacing}, that for every strictly increasing sequence~$(\vs_i)_{i \in \N}$ in~$\vN$, there exists a corresponding sequence in~$\cP$ such that they together satisfy~\cref{item:Orientation} and~\cref{item:DistinguishingPairwise} (and thus also~\cref{item:DistinguishingSequence} by~\cref{lem:InterlacingDistinguishingSequence}) -- if every pair~$\vs_i < \vs_{i+1}$ is~$\cP$-relevant in a strong form.
To make this precise, recall that a separation of a graph~$G$ is~\emph{$\cP$-relevant} for a set~$\cP$ of pre-tangles in~$G$ if it efficiently distinguishes some two pre-tangles in~$\cP$.
We now call a pair~$\vs < \vt$ of two oriented separations of~$G$ \emph{strongly~$\cP$-relevant} if there exist pre-tangles~$O, P, Q \in \cP$ such that
\begin{itemize}
    \item $s$ efficiently distinguishes~$O$ and~$P$ with~$\vs \in P$, and
    \item $t$ efficiently distinguishes~$P$ and~$Q$ with~$\tv \in P$.
\end{itemize}
A strictly increasing sequence~$(\vs_i)_{i \in \N}$ of separations of~$G$ is then called \emph{strongly~$\cP$-relevant} if every two consecutive separations~$\vs_i < \vs_{i+1}$ are strongly~$\cP$-relevant.

\begin{lemma} \label{lem:InterlacingStronglyRelevant}
    Let~$\cP$ be a set of pre-tangles in a graph~$G$, let~$N$ be a nested set of separations of~$G$ which efficiently distinguishes~$\cP$, and let~$(\vs_i)_{i \in \N}$ be a strictly increasing sequence of separations in~$\vN$.
    If~$(\vs_i)_{i \in \N}$ is strongly~$\cP$-relevant, then there exists a sequence~$(P_i)_{i \in \N}$ in~$\cP$ which together with~$(\vs_i)_{i \in \N}$ satisfies~\cref{item:Orientation} and~\cref{item:DistinguishingPairwise}.
\end{lemma}
\begin{proof}
    Since~$(\vs_i)_{i \in \N}$ is strongly~$\cP$-relevant, there exist, for every~$i \in \N$, pre-tangles~$O_{i+1}$, $P_{i+1}$ and~$Q_{i+1}$ in~$\cP$~with \looseness=-1
    \begin{enumerate}
        \item \label{prop:Oi+1Pi+1}~$s_{i}$ efficiently distinguishes~$O_{i+1}$ and~$P_{i+1}$ with~$\vs_{i} \in P_{i+1}$, and
        \item \label{prop:Pi+1Qi+1}~$s_{i+1}$ efficiently distinguishes~$P_{i+1}$ and~$Q_{i+1}$ with~$\sv_{i+1} \in P_{i+1}$.
    \end{enumerate}
    Set~$P_0 := O_1$.
    Then the sequence~$(P_i)_{i \in \N}$ together with~$(\vs_i)_{i \in \N}$ immediately satisfies~\cref{item:Orientation} by~\labelcref{prop:Oi+1Pi+1,prop:Pi+1Qi+1}.
    We prove that they also satisfy~\cref{item:DistinguishingPairwise}.

    By~\cref{item:Orientation}, $s_i$ distinguishes~$P_i$ and~$P_{i+1}$.
    Thus, it remains to show that~$s_i$ does so efficiently.
    Suppose for a contradiction that this is not the case.
    We then have~$i \ge 1$ by the definition of~$P_0$ as~$O_1$ and~\cref{prop:Oi+1Pi+1}.
    Since~$N$ efficiently distinguishes~$\cP$, there is a separation~$r \in N$ that efficiently distinguishes~$P_i$ and~$P_{i+1}$; in particular~$|r| < |s_i|$.
    Let~$\vr$ be the orientation of~$r$ in~$P_{i+1}$, then~$\rv \in P_i$.
    Both~$\vs_i$ and~$\vr$ are contained \linebreak in~$P_{i+1} \cap \vN$, so~\cref{obs:DistinguisherFormChain} implies that either~$\vr < \vs_{i}$ or~$\vs_{i} < \vr$.
    We show that both cases lead to contradictions.

    First, suppose~$\vr < \vs_i$.
    Since~$i \ge 1$, the pre-tangle~$Q_i$ is defined, and we have~$\vr \in Q_i$, as~$\vs_i \in Q_i$ by the choice of~$Q_i$ and~$Q_i$ is consistent.
    But~$\rv \in P_i~$, so~$r$ distinguishes~$P_i$ and~$Q_i$.
    This then contradicts~\cref{prop:Pi+1Qi+1} as~$|r| < |s_i|$.

    Secondly, assume that~$\vs_i < \vr$.
    This implies~$\rv < \sv_i$.
    We proceed analogously to the first case, but for~$O_{i+1}$.
    We have~$\sv_i \in O_{i+1}$ by the choice of~$O_{i+1}$, and hence~$\rv \in O_{i+1}$, since~$O_{i+1}$ is consistent.
    But~$\vr \in P_{i+1}$, so~$r$ distinguishes~$O_{i+1}$ and~$P_{i+1}$ which contradicts~\cref{prop:Oi+1Pi+1} due to~$|r| < |s_i|$.
\end{proof}

The third~\cref{lem:InterlacingStrictlyIncreasingOrder} asserts, in the setting of~\cref{lem:ProfileInterlacing}, that a strictly increasing sequence of~$\cP$-relevant separations with strictly increasing orders is interlaced with a strongly~$\cP$-relevant strictly increasing sequence. \looseness=-1

\begin{lemma} \label{lem:InterlacingStrictlyIncreasingOrder}
    Let~$\cP$ be a set of pre-tangles in a graph~$G$, and let~$N$ be a nested set of~$\cP$-relevant proper separations of~$G$ which efficiently distinguishes~$\cP$.
    Let~$(\vs_i)_{i \in \N}$ be a strictly increasing sequence in~$\vN$.
    If the orders of the~$s_i$ are strictly increasing, then there exists a strictly increasing sequence~$(\vs_i')_{i \in \N}$ in~$\vN$ that contains~$(\vs_i)_{i \in \N}$ as a subsequence and is strongly~$\cP$-relevant.
\end{lemma}
\begin{proof}
    We recursively find a strictly increasing function~$f\colon \N \to \N$ and a strictly increasing sequence~$(\vs_i')_{i \in \N}$ in~$\vN$ which is strongly~$\cP$-relevant such that for every~$k \in \N$, we have
    \begin{enumerate}
        \item \label{prop:siissubsequencestrong}~$\vs_{f(k)}' = \vs_k$, and
        \item \label{prop:stronglyOrelevant} the pair~$\vs_i' < \vs_{i+1}'$ is strongly~$\cP$-relevant for all~$i < f(k)$.
    \end{enumerate}
    
    \noindent Clearly, such a sequence~$(\vs_i')_{i \in \N}$ is as desired.
    
    To start the recursion, we set~$f(0) := 0$ and~$\vs_0' := \vs_0$.
    Let~$\ell \in \N$ such that we have already defined~$f(0), \dots, f(\ell)$ and~$(\vs_i')_{i \le f(\ell)}$ which satisfy properties \cref{prop:siissubsequencestrong} and~\cref{prop:stronglyOrelevant} for every~$k \le \ell$.
    We now proceed by defining~$f(\ell+1) > f(\ell)$ and~$\vs_i'$ for the remaining~$i = f(\ell) + 1, \dots, f(\ell+1)$ such that~$(\vs_i')_{i \le f(\ell+1)}$ satisfies the properties \cref{prop:siissubsequencestrong} and~\cref{prop:stronglyOrelevant} for every~$k \le \ell+1$.

    The recursion step roughly works as follows.
    Appending the separation~$\vs_{\ell+1}$ to~$(\vs_i')_{i \le f(\ell)}$ would immediately yield~\cref{prop:siissubsequencestrong}.
    Then either the pair~$\vs_\ell < \vs_{\ell+1}$ is strongly~$\cP$-relevant itself, which implies~\cref{prop:stronglyOrelevant}, or we will insert an additional separation~$\vr$ in between~$\vs_\ell$ and~$\vs_{\ell+1}$ to then get two strongly~$\cP$-relevant pairs~$\vs_\ell < \vr$ and~$\vr < \vs_{\ell+1}$ and hence~\cref{prop:stronglyOrelevant}, which will then conclude the proof.

    Since both~$\vs_\ell$ and~$\vs_{\ell+1}$ are~$\cP$-relevant, there exist pre-tangles~$P_j$ and~$Q_j$ in~$\cP$ for~$j = \ell, \ell+1$ such that~$s_j$ efficiently distinguishes~$P_j$ and~$Q_j$ with~$\vs_j \in Q_j$.
    Since we have~$|s_\ell| < |s_{\ell+1}|$ by assumption, the pre-tangles~$P_{\ell+1}$ and~$Q_{\ell+1}$ in~$G$ orient not only~$s_{\ell+1}$, but also~$s_\ell$, and they do so in the same way as~$s_{\ell+1}$ efficiently distinguishes them.
    We have~$\vs_\ell \in Q_{\ell+1}$, as~$Q_{\ell+1}$ is consistent and~$\vs_\ell < \vs_{\ell+1} \in Q_{\ell+1}$, and hence~$\vs_\ell \in P_{\ell+1}$.
    Altogether, $s_\ell$ distinguishes~$P_\ell$ and~$P_{\ell+1}$.
    
    First, let us assume that~$s_\ell$ distinguishes~$P_\ell$ and~$P_{\ell+1}$ efficiently.
    Then the pre-tangles~$P_\ell, P_{\ell+1}, Q_{\ell+1} \in \cP$ witness the strong~$\cP$-relevance of the pair~$\vs_\ell < \vs_{\ell+1}$.
    So in this case, we set~$f(\ell+1) := f(\ell)+1$ and~$\vs'_{f(\ell+1)} := \vs_{\ell+1}$.
    Then~$(\vs_i')_{i \le f(\ell+1)}$ satisfies \cref{prop:siissubsequencestrong} and~\cref{prop:stronglyOrelevant} for every~$k \le \ell+1$, completing the recursion step. \looseness=-1
    
    Secondly, we assume that~$s_\ell$ does not distinguish~$P_\ell$ and~$P_{\ell+1}$ efficiently.
    Since~$N$ efficiently distinguishes~$\cP$, there exists a separation~$r \in N$ which distinguishes~$P_\ell$ and~$P_{\ell+1}$ efficiently; in particular, $|r| < |s_\ell|$.
    Let~$\vr$ be the orientation of~$r$ in~$P_{\ell+1}$ and~$\rv$ the one in~$P_\ell$.
    In the remainder of the proof, we show that setting~$f(\ell+1) := f(\ell) + 2$, $\vs'_{f(\ell)+1} := \vr$ and~$\vs'_{f(\ell)+2} := \vs_{\ell+1}$ yields the desired~$(\vs_i')_{i \le f(\ell+1)}$ satisfying~\cref{prop:siissubsequencestrong} and~\cref{prop:stronglyOrelevant} for all~$k \le \ell + 1$.
    To do so, it remains to show that~$\vs_\ell < \vr < \vs_{\ell+1}$ and that both pairs~$\vs_\ell < \vr$ and~$\vr < \vs_{\ell+1}$ are strongly~$\cP$-relevant.
    
    Let us begin with~$\vs_\ell < \vr$.
    Recall that both~$r, s_\ell \in N$ distinguish~$P_\ell$ and~$P_{\ell+1}$ with~$\vr, \vs_\ell \in P_{\ell+1}$.
    So by~\cref{obs:DistinguisherFormChain}, we have either~$\vr < \vs_\ell$ or~$\vs_\ell < \vr$; we show that the former cannot occur.
    The pre-tangle~$Q_\ell$ in~$G$ orients~$r$ as it orients~$s_\ell$ and~$|r| < |s_\ell|$.
    So if~$\vr < \vs_\ell$, then we have~$\vr \in Q_\ell$ since~$\vs_\ell \in Q_\ell$ and~$Q_\ell$ is consistent.
    But~$\rv \in P_\ell$, so~$r$ distinguishes~$P_\ell$ and~$Q_\ell$ and~$|r| < |s_\ell|$ which contradicts that~$s_\ell$ efficiently distinguishes~$P_{\ell}$ and~$Q_\ell$ by their choice.
    Thus, we must have~$\vs_\ell < \vr$, as desired.

    Next, we claim that~$\vr < \vs_{\ell+1}$.
    Recall that~$r$ and~$s_{\ell}$ are in~$N$ and hence nested.
    By the definition of nested, it suffices to show that neither of~$\rv < \vs_{\ell+1}$, $\vs_{\ell+1} < \vr$ and~$\vs_{\ell+1} < \rv$ holds; we address all these cases separately. \looseness=-1

    We cannot have
    $\vs_{\ell+1} < \vr$, 
    as we would otherwise have a contradiction to the consistency of the pre-tangle~$P_{\ell+1}$ in~$G$ which by choice already contains both~$\vr$ and
    $\sv_{\ell+1}$.
    For
    $\rv < \vs_{\ell+1}$, 
    observe that~$r$ cannot distinguish~$P_{\ell+1}$ and~$Q_{\ell+1}$ as they are efficiently distinguished by~$s_{\ell+1}$ and~$|r| < |s_{\ell+1}|$.
    Thus, $\vr \in P_{\ell+1}$ 
    yields~$\vr \in Q_{\ell+1}$.
    But by the choice of $Q_{\ell+1}$ it also contains~$\vs_{\ell+1}$  
    which together with $\rv < \vs_{\ell+1}$ then contradicts the consistency of~$Q_{\ell+1}$.
    Finally, if~$\vs_{\ell+1} < \rv$, then this implies~$\vs_{\ell} < \rv$ since~$\vs_\ell < \vs_{\ell+1}$.
    But we also have~$\vs_{\ell} < \vr$ as shown above, a contradiction to~$s_{\ell}$ being proper.
    
    It remains to show that both pairs~$\vs_\ell < \vr$ and~$\vr < \vs_{\ell+1}$ are strongly~$\cP$-relevant.
    More precisely, we claim that this is witnessed by the pre-tangles~$P_\ell, Q_\ell, P_{\ell+1} \in \cP$ and the pre-tangles~$Q_\ell, P_{\ell+1}, Q_{\ell+1} \in \cP$, respectively.
    To show this, it suffices to prove that~$r$ efficiently distinguishes~$Q_\ell$ and~$P_{\ell+1}$ with~$\rv \in Q_\ell$ and~$\vr \in P_{\ell+1}$.
    
    Since~$s_\ell$ efficiently distinguishes~$P_\ell$ and~$Q_\ell$ and~$|r| < s_{\ell}$, the pre-tangle~$Q_\ell$ in~$G$ orients~$r$ in the same way as~$P_\ell$.
    Hence, $\rv \in P_\ell$ yields~$\rv \in Q_\ell$.
    This implies that~$r$ distinguishes~$Q_\ell$ and~$P_{\ell+1}$ with~$\rv \in Q_\ell$ and~$\vr \in P_{\ell+1}$, and we claim that~$r$ does so efficiently.
    Indeed, the pre-tangles~$Q_\ell$ and~$P_\ell$ in~$G$ orient every separation of order less than~$|s_\ell|$ in the same way, since~$s_\ell$ efficiently distinguishes them.
    Hence, every separation of order less than~$|s_\ell|$ that distinguishes~$Q_\ell$ and~$P_{\ell+1}$ also distinguishes~$P_\ell$ and~$P_{\ell+1}$.
    But~$r$ efficiently distinguishes~$P_\ell$ and~$P_{\ell+1}$ by the choice of~$r$ and~$|r| < |s_\ell|$, so there is no separation of order less than~$r$ distinguishing~$Q_\ell$ and~$P_{\ell+1}$, as desired.
\end{proof}

The fourth~\cref{lem:InterlacingThinOutForStrictlyIncreasing} shows that by transition to suitable subsequences, we can maintain both~\cref{item:Orientation} and~\cref{item:DistinguishingSequence} while passing from `the orders cofinitely exceed every integer' to `strictly increasing orders'.

\begin{lemma} \label{lem:InterlacingThinOutForStrictlyIncreasing}
    Let~$(\vs_i)_{i \in \N}$ be a strictly increasing sequence of separations of a graph~$G$, and let~$(P_i)_{i \in \N}$ be a sequence of pre-tangles in~$G$ such that~\cref{item:Orientation} and~\cref{item:DistinguishingSequence} hold.
    If~$(|s_i|)_{i \in \N}$ cofinitely exceeds every integer, then there exists a subsequence~$(\vs_{j(i)})_{i \in \N}$ of strictly increasing orders such that it still satisfies~\cref{item:Orientation} and~\cref{item:DistinguishingSequence} together with~$(P_{j(i)})_{i \in \N}$.
\end{lemma}
\begin{proof}
    We inductively define~$j(i)$ for all~$i \in \N$.
    To do so, note that since~$(|s_\ell|)_{\ell \in \N}$ cofinitely exceeds every integer, there exist only finitely many~$s_\ell$ of some fixed order.
    For~$i = 0$ let~$K_0$ be the minimal order of any~$s_\ell$, and set~$j(0)$ to be the maximal~$\ell \in \N$ with~$|s_\ell| = K_0$; this maximum exists since there are only finitely many such~$\ell$, as noted above.
    For~$i \ge 1$ we analogously let~$K_i$ be the minimal order of any~$s_\ell$ with~$\ell > j(i-1)$ and set~$j(i)$ to be the maximal~$\ell \in \N$ with~$|s_\ell| = K_i$.
    With this construction we ensure that $|s_\ell| > |s_{j(i)}|$ for all~$\ell > j(i)$.
    
    First, we verify~\cref{item:Orientation} for~$(\vs_{j(i)})_{i \in \N}$ and~$(P_{j(i)})_{i \in \N}$.
    The original sequences satisfy~\cref{item:Orientation} which immediately yields~$\sv_{j(i)} \in P_{j(i)}$ and~$\vs_{j(i+1)-1} \in P_{j(i+1)}$.
    By construction, $|s_{j(i)}| \le |s_{j(i+1)-1}|$, so the pre-tangle~$P_{j(i+1)}$ in~$G$ orients~$s_{j(i)}$.
    Since~$\vs_{j(i)} \le \vs_{j(i+1)-1} \in P_{j(i+1)}$, the consistency of~$P_{j(i+1)}$ yields $\vs_{j(i)} \in P_{j(i+1)}$.
    This completes the verification of~\cref{item:Orientation}. \looseness=-1
    
    Secondly, we prove~\cref{item:DistinguishingSequence}, and we do so by checking~\cref{item:DistinguishingPairwise} which suffices by~\cref{lem:InterlacingDistinguishingSequence}.
    Applying~\cref{item:DistinguishingSequence} to~$j(i) < j(i+1)$ in the original sequence yields that every minimal-order separation \linebreak among~$\{s_{j(i)}, \dots, s_{j(i+1)-1}\}$ efficiently distinguishes~$P_{j(i)}$ and~$P_{j(i+1)}$.
    But by the definition of~$j(i)$, $s_{j(i)}$ is the unique such separation.
    Hence, \cref{item:DistinguishingPairwise} holds.
\end{proof}

With~\cref{lem:InterlacingDistinguishingSequence,lem:InterlacingStronglyRelevant,lem:InterlacingStrictlyIncreasingOrder,lem:InterlacingThinOutForStrictlyIncreasing} at hand, we are now ready to prove~\cref{lem:ProfileInterlacing}.

\begin{proof}[Proof of~\cref{lem:ProfileInterlacing}]
    Since~$((A_i, B_i))_{i \in \N}$ is non-exhaustive, \cref{lem:InterlacingLemma} and~\cref{lem:BoundedOrder} together imply that~$(|A_i, B_i|)_{i \in \N}$ cofinitely exceeds every integer.
    In particular, $((A_i, B_i))_{i \in \N}$ contains a subsequence whose orders are even strictly increasing.
    Note that if a sequence of separations is interlaced with a subsequence of~$((A_i, B_i))_{i \in \N}$, then it is also interlaced with~$((A_i, B_i))_{i \in \N}$ itself.
    So by transition to a subsequence, we may assume without loss of generality that~$(|A_i, B_i|)_{i \in \N}$ is strictly increasing.
    
    Thus, we can apply~\cref{lem:InterlacingStrictlyIncreasingOrder} to~$((A_i, B_i))_{i \in \N}$ and obtain a strictly increasing sequence~$((A_i', B_i'))_{i \in \N}$ in~$\vN$ which is strongly~$\cP$-relevant and contains~$((A_i, B_i))_{i \in \N}$ as a subsequence; in particular, $((A_i, B_i))_{i \in \N}$ and~$((A_i', B_i'))_{i \in \N}$ are interlaced.
    \cref{lem:InterlacingStronglyRelevant} then yields a corresponding sequence~$(P_i)_{i \in \N}$ in~$\cP$ such that~\labelcref{item:Orientation,item:DistinguishingPairwise} hold.
    \cref{lem:InterlacingDistinguishingSequence} finally implies that~\cref{item:DistinguishingSequence} holds as well.

    The `moreover'-part can be obtained by applying~\cref{lem:InterlacingThinOutForStrictlyIncreasing} to~$((A'_i, B'_i))_{i \in \N}$ and~$(P_i)_{i \in \N}$.
    In order to apply~\cref{lem:InterlacingThinOutForStrictlyIncreasing}, it remains to check that~$(|A_i', B_i'|)_{i \in \N}$ cofinitely exceeds every integer.
    By~\cref{lem:InterlacingLemma}, both~$((A_i, B_i))_{i \in \N}$ and~$((A_i', B_i'))_{i \in \N}$ have the same limit; in particular, $((A_i', B_i'))_{i \in \N}$ is non-exhaustive as~$((A_i, B_i))_{i \in \N}$ is.
    Then~\cref{lem:InterlacingLemma} and~\cref{lem:BoundedOrder} imply that~$(|A_i', B_i'|)_{i \in \N}$ cofinitely exceed every integer, as desired.
\end{proof}

We remark that~\cref{lem:ProfileInterlacing} and its proof are purely about sequences of separations and do not refer to the underlying graph in any way.
Thus, \cref{lem:ProfileInterlacing} can also be transferred to the more general framework of abstract separation systems which are extensively described in~\cite{ASS}.
There, we would then consider a universe of separations equipped with a (not necessarily submodular) order function and pre-tangles in this universe which are efficiently distinguished by a regular nested set of separations.

\section{Proofs of Theorem 1 and 2} \label{sec:ProofMainTheorem}

In this section, we prove our main results, \cref{main:Theorem} and its structural version~\cref{main:Technical}, from which we will deduce~\cref{main:Theorem}.
To prove~\cref{main:Technical}, we will apply the interlacing theorem, \cref{lem:ProfileInterlacing}, to analyse the limit separation of a non-exhaustive strictly increasing sequence of separations in a tree of tangles, and we then exhibit a thick end in the separator of the limit separation.

Let~$\omega$ be an end of a graph~$G$.
Recall that for every finite set~$X$ of vertices of~$G$, there exists a unique component of~$G-X$ which contains cofinitely many vertices of some (or equivalently every) ray in~$\omega$. 
Thus, the end~$\omega$ \emph{induces} a pre-tangle~$P_\omega$ in~$G$: it orients every finite-order separation of~$G$ to the side which contains cofinitely many vertices of some (or equivalently every) ray in~$\omega$~\cite{EndsAndTangles}.

Recall that an end~$\omega$ is \emph{thick} if it contains infinitely many disjoint rays, and if~$\omega$ is not thick, then it is \emph{thin}.
The next lemma shows that the pre-tangle~$P_\omega$ induced by a thin end~$\omega$ can be particularly well distinguished from all other pre-tangles in~$G$.

\begin{lemma} \label{lem:ThinEndSequence}
    Let~$\omega$ be a thin end of a graph~$G$, and let~$P_\omega$ be the pre-tangle in~$G$ that is induced by~$\omega$.
    Then there exists~$K \in \N$ such that every pre-tangle~$P$ in~$G$ with~$P \nsubseteq P_\omega$ is distinguished from~$P_\omega$ by a~separation of order at most~$K$.
\end{lemma}
\begin{proof}
    Since~$\omega$ is thin, there exists~$K \in \N$ and an exhaustive strictly increasing sequence~$((A_i, B_i))_{i \in \N}$ of separations of order at most~$K$ such that~$(A_i, B_i) \in P_\omega$ and~$G[B_i]$ is connected for all~$i \in \N$~\cite{gollin2022characterising}*{Corollary~5.5}.
    Let~$P$ be a pre-tangle in~$G$ with~$P \nsubseteq P_\omega$. 
    Since~$P_\omega$ orients all finite-order separations of~$G$, there thus is some finite-order separation~$\{C, D\}$ of~$G$ distinguishing~$P$ and~$P_\omega$ in that~$(C, D) \in P_\omega$ and~$(D, C) \in P$.
    If~$\{C, D\}$ has order at most~$K$, then it is as desired; so suppose that~$\{C, D\}$ has order greater than~$K$. \looseness=-1

    Since the sequence~$((A_i, B_i))_{i \in \N}$ is exhaustive and all the~$G[B_i]$ are connected, \cref{lem:ExhaustiveFiniteOrderBelowLimit} yields $I \in \N$ such that either~$(C, D) < (A_I, B_I)$ or~$(D, C) < (A_I, B_I)$.
    We claim that~$\{A_I, B_I\}$ is as desired; since~$\{A_I, B_I\}$ has order at most~$K$, it remains to show that~$\{A_I, B_I\}$ distinguishes~$P$ and~$P_\omega$. 
    Since~$P_\omega$ is consistent and both~$(C, D)$ and~$(A_I, B_I)$ are contained in~$P_\omega$, we must have~$(C, D) < (A_I, B_I)$.
    The separation~$\{A_I, B_I\}$ is also oriented by the pre-tangle~$P$ in~$G$ since it has order at most~$K$ and hence order less than~$(D, C) \in P$.
    The consistency of~$P$ then implies~$(B_I, A_I) \in P$ since~$(B_I, A_I) < (D, C) \in P$.
    Thus, $\{A_I, B_I\}$ distinguishes~$P$ and~$P_\omega$, as desired. \looseness=-1
\end{proof}

A \emph{comb} is a union of a ray~$R$ with infinitely many disjoint finite paths which have precisely their first vertex on~$R$; we call the last vertices of these paths the \emph{teeth} of the comb and refer to~$R$ as its \emph{spine}.
The star-comb lemma~\cite{DiestelBook16}*{Proposition 8.2.2} implies the following simple observation about combs and infinite sets of vertices in locally finite graphs.

\begin{lemma} \label{lem:StarComb}
    If~$U$ is an infinite set of vertices of a locally finite graph~$G$, then~$G$ contains a comb with teeth in~$U$ as a subgraph.
    \qed
\end{lemma}

Consequently, we say that an end~$\omega$ of the locally finite~$G$ is 
in the \emph{closure} of a set~$U \subseteq V(G)$ if there exists a comb with teeth in~$U$ whose spine is in~$\omega$.\footnote{An end of a locally finite graph $G$ is 
in the closure of
a set~$U$ if and only if it is contained in its topological closure~$\bar{U}$ in the Freudenthal compactification~$|G|$ of $G$ (cf.~\cite{DiestelBook16}*{Section~8.6}).}
It is immediate from this definition that if~$\omega$ is 
in the closure of~$U$ and~$(A, B)$ is a separation in the pre-tangle~$P_\omega$ in~$G$ induced by~$\omega$, then infinitely many vertices of~$U$ are contained in~$B$.
Similarly, if~$\{A, B\}$ is a finite-order separation of~$G$ and the end~$\omega$ is 
in the closure of~$B$, then~$(A, B) \in P_\omega$.

Recall that if a strictly increasing sequence of tight separations of a connected locally finite graph is non-exhaustive, then its limit separation has infinite order by~\cref{lem:LimitSepHasInfiniteOrder}.
Then \cref{lem:StarComb} implies that some end of the graph is 
in the closure of the separator of such a limit separation.
The next lemma shows that this end is unique if the sequence is contained in a nested set distinguishing all the pre-tangles induced by ends in the closure of the separator.

\begin{lemma} \label{lem:UniqueEnd}
    Let~$N$ be a nested set of separations of a connected locally finite graph~$G$, let~$((A_i, B_i))_{i \in \N}$ be a strictly increasing sequence of tight separations in~$\vN$, and let~$(A, B)$ be its limit.
    If~$((A_i, B_i))_{i \in \N}$ is non-exhaustive and~$N$ distinguishes all pre-tangles in~$G$ that are induced by ends 
    in the closure of~$A \cap B$, then precisely one end of~$G$ is 
    in the closure of~$A \cap B$.
\end{lemma}
\begin{proof}
    Since~$((A_i, B_i))_{i \in \N}$ is non-exhaustive, $A \cap B$ is infinite by~\cref{lem:LimitSepHasInfiniteOrder}.
    So at least one end of~$G$ is 
    in the closure of~$A \cap B$ by~\cref{lem:StarComb}.
    Suppose for a contradiction that two distinct ends~$\omega$ and~$\omega'$ of~$G$ are 
    in the closure of~$A \cap B$, and let~$P_\omega$ and~$P_{\omega'}$ be the pre-tangles in~$G$ that are induced by~$\omega$ and~$\omega'$, respectively.
    By assumption, there exists a finite-order separation~$\{C, D\} \in N$ which distinguishes~$P_\omega$ and~$P_{\omega'}$ with~$(D, C) \in P_\omega$ and~$(C, D) \in P_{\omega'}$.
    It suffices to show that~$\{C, D\}$ must cross~$\{A, B\}$.
    Indeed, \cref{lem:FiniteOrderCrossLimit} then implies that~$\{C, D\}$ crosses some~$\{A_i, B_i\}$ which contradicts the nestedness of~$N$ in which they are both contained.

    To show that~$\{C, D\}$ crosses~$\{A, B\}$, observe that there exist infinitely many elements of~$A \cap B$ in~$C$, since~$\omega$ is 
    in the closure of~$A \cap B$.
    This implies~$(C \setminus D) \cap (A \cap B) \neq \emptyset$ as~$C \cap D$ is finite.
    Arguing analogously for~$(C, D) \in P_{\omega'}$, we obtain~$(D \setminus C) \cap (A \cap B) \neq \emptyset$.
    But then~$\{A, B\}$ and~$\{C, D\}$ cross by~\cref{lem:NestedWithCorners}.
\end{proof}

We now use the interlacing theorem, \cref{lem:ProfileInterlacing}, to show~\cref{thm:ThickEnd}.
It asserts that given a non-exhaustive strictly increasing sequence in a nested set distinguishing all pre-tangles of a locally finite graph, the (by~\cref{lem:UniqueEnd} unique) end which is 
in the closure of the separator of its limit separation cannot be thin.

\begin{theorem} \label{thm:ThickEnd}
    Let~$\cP$ be a set of pre-tangles in a connected locally finite graph~$G$, and let~$N$ be a nested set of~$\cP$-relevant tight separations which efficiently distinguishes~$\cP$.
    Let~$((A_i, B_i))_{i \in \N}$ be a strictly increasing sequence in~$\vN$, and let~$(A, B)$ be its limit.

    If~$((A_i, B_i))_{i \in \N}$ is non-exhaustive
    and~$\cP$ contains all the pre-tangles in~$G$ which are induced by ends
    in the closure of~$A \cap B$, then there is precisely one end~$\omega$ in the closure of $A \cap B$ and it is thick.
\end{theorem}
\begin{proof}
    The assumptions on~$N$ and~$((A_i, B_i))_{i \in \N}$ allow us to apply~\cref{lem:ProfileInterlacing} to obtain a strictly increasing sequence~$((A_i', B_i'))_{i \in \N}$ in~$\vN$ and a corresponding sequence~$(P_i)_{i \in \N}$ in~$\cP$ satisfying \cref{item:Orientation} and~\cref{item:DistinguishingSequence}.
    Furthermore, the sequence~$((A_i', B_i'))_{i \in \N}$ is interlaced with~$((A_i, B_i))_{i \in \N}$, which implies by \cref{lem:InterlacingLemma} that~$(A, B)$ is also the limit of~$((A_i', B_i'))_{i \in \N}$.
    In particular, this sequence is non-exhaustive.
    By the `moreover'-part of \cref{lem:ProfileInterlacing}, we may additionally assume that~$(|A_i', B_i'|)_{i \in \N}$ is strictly increasing and then~$(A_i', B_i')$ efficiently distinguishes~$P_i$ from every~$P_j$ with~$j > i$.

    By~\cref{lem:UniqueEnd}, there is a unique end $\omega$ in the closure of $A \cap B$.
    Let~$P_\omega$ be the pre-tangle in~$G$ which is induced by  
    $\omega$.
    The following claim shows that every separation in our new sequence~$((A_i', B_i'))_{i \in \N}$ efficiently distinguishes~$P_\omega$ from the respective~$P_i$.
    \begin{claim*}
        Every~$\{A_i', B_i'\}$ efficiently distinguishes~$P_i$ and~$P_\omega$.
    \end{claim*}
    \begin{claimproof}
        We first show that~$\{A_i', B_i'\}$ distinguishes~$P_i$ and~$P_\omega$.
        Indeed, $(B_i', A_i') \in P_i$ by~\cref{item:Orientation}.
        To see~$(A_i', B_i') \in P_\omega$, recall that~$(A_i', B_i') \le (A, B)$.
        Thus, $A \cap B$ is contained in~$B_i'$.
        Since~$\omega$ is 
        in the closure of~$A \cap B$, it is thus 
        in the closure of~$B_i'$ as well, and this implies~$(A_i', B_i') \in P_\omega$ by definition.

        Suppose that~$\{A_i', B_i'\}$ does not distinguish~$P_i$ and~$P_\omega$ efficiently.
        The nested set~$N$ efficiently distinguishes~$\cP$, so~$N$ contains a separation~$\{C, D\}$ which efficiently distinguishes~$P_i$ and~$P_\omega$ in that $(C, D) \in P_\omega$ and~$(D, C) \in P_i$.
        In particular, we have~$|C, D| < |A_i', B_i'|$. 
        We now claim that~$(A_i', B_i') < (C, D)$ and that there exists some integer~$I > i$ with~$(C, D) < (A_I', B_I')$.
        
        First, let us show that~$(A_i', B_i') < (C, D)$.
        The nested separations~$\{C, D\}$ and~$\{A_i', B_i'\}$ both distinguish~$P_i$ and~$P_\omega$ with~$(C, D), (A_i', B_i') \in P_\omega$ as above.
        Thus, \cref{obs:DistinguisherFormChain} yields that either~$(A_i', B_i') < (C, D)$ or~$(C, D) < (A_i', B_i')$.
        Suppose for a contradiction that~$(C, D) < (A_i', B_i')$.
        Recall that~$(A_i', B_i') \in P_{i+1}$ by~\cref{item:Orientation}.
        The consistency of the pre-tangle~$P_{i+1}$ in~$G$ thus yields~$(C, D) \in P_{i+1}$, since~$|C, D| < |A_i', B_i'|$.
        We also have~$(D, C) \in P_i$ by the choice of~$\{C, D\}$, so~$\{C, D\}$ distinguishes~$P_i$ and~$P_{i+1}$.
        But~$|C, D| < |A_i', B_i'|$ which contradicts \cref{item:DistinguishingSequence} for~$i < i+1$.
        Thus, we must have~$(A_i', B_i') < (C, D)$, as desired.

        Secondly, we find an integer~$I > i$ with~$(C, D) < (A_I', B_I')$.
        By~\cref{lem:FiniteOrderBelowLimit}, it suffices to show that~$(C, D) \le (A, B)$.
        Since~$\{C, D\}$ and all the~$\{A_j', B_j'\}$ are contained in the nested set~$N$, \cref{lem:FiniteOrderCrossLimit} implies that~$\{C, D\}$ has to be nested with~$\{A, B\}$ as well.
        By the choice of~$\{C, D\}$, we have $(C, D) \in P_\omega$, so~$D$ contains infinitely many elements of~$A \cap B$, and since~$\{C, D\}$ has finite order, we have~$(D \setminus C) \cap (A \cap B) \neq \emptyset$.
        Hence, either~$(C,D) \le (A,B)$ or~$(C,D) \le (B,A)$ holds by~\cref{lem:NestedWithCorners} because~$\{C, D\}$ and~$\{A, B\}$ are nested.
        Suppose for a contradiction that~$(C,D) \le (B,A)$.
        Then~$(A_i', B_i') < (C, D) \le (B, A)$, but we also have~$(A_i', B_i') \le (A,B)$ which contradicts the fact that~$\{A_i', B_i'\}$ is tight and hence proper.

        We complete the proof of the claim by obtaining a contradiction from~$(A_i', B_i') < (C, D) < (A_I', B_I')$.
        By~\cref{item:Orientation}, we have that~$(B_i', A_i') \in P_i$ and~$(A_I', B_I') \in P_{I+1}$.
        Since~$|C, D| < |A_i', B_i'| < |A_I', B_I'|$, both pre-tangles~$P_i$ and~$P_{I+1}$ in~$G$ orient~$\{C, D\}$.
        The consistency of the pre-tangles~$P_i$ and~$P_{I+1}$ together with $(A_i', B_i') < (C, D) < (A_I', B_I')$ then yields~$(D, C) \in P_i$ and~$(C, D) \in P_{I+1}$; in particular, $\{C, D\}$ distinguishes~$P_i$ and~$P_{I+1}$.
        But~$|C, D| < |A_i', B_i'|$, which contradicts that~$\{A_i', B_i'\}$ efficiently distinguishes~$P_i$ and~$P_{I+1}$ by the choice of the sequence~$((A_i',B_i'))_{i \in \N}$.
    \end{claimproof}
    
    Now suppose for a contradiction that~$\omega$ is thin.
    Then~\cref{lem:ThinEndSequence} yields some~$K \in \N$ such that~$P_\omega$ is distinguished from every other pre-tangle~$P \nsubseteq P_\omega$ in~$G$ by a separation of order at most~$K$.
    But the orders of the~$\{A_i', B_i'\}$ are strictly increasing, so there exists~$j \in \N$ such that~$(A_j', B_j')$ has order greater than~$K$.
    By the claim, this~$\{A_j', B_j'\}$ efficiently distinguishes~$P_j$ and~$P_\omega$, which then contradicts the choice of~$K$ and hence the thinness of~$\omega$.
    So~$\omega$ is a thick end of~$G$, as desired.
\end{proof}

\noindent Instead of finishing the proof of~\cref{thm:ThickEnd} by a contradiction in terms of cardinality via~\cref{lem:ThinEndSequence}, one can also perform a structural analysis to obtain that the thickness of~$\omega$ is even witnessed by an infinite family of disjoint rays in~$G[B \setminus A]$ (see~\cref{app:InfRaysBeyondLimit}).

It remains to deduce our main results, \cref{main:Theorem,main:Technical}, by applying~\cref{thm:ThickEnd} to a tree of tangles of the graph~$G$.
To do so, recall that a~\emph{tangle~$P$ in~$G$} is a pre-tangle in~$G$ such that for every three -- not necessarily distinct -- separations~$(A_1, B_1), (A_2, B_2), (A_3, B_3) \in P$, we have~$G[A_1] \cup G[A_2] \cup G[A_3] \neq G$.
We remark that every pre-tangle induced by some end is in fact a tangle.
Recall further that a \emph{tree of tangles} of~$G$ is a nested set of separations which efficiently distinguishes the set~$\cP$ of all the tangles in~$G$ and which contains only~$\cP$-relevant separations.

The next lemma asserts that every separation in a tree of tangles, and even every~$\cP$-relevant separation, is tight.

\begin{lemma}[\cite{InfiniteSplinters}*{Lemma~6.1 and its proof}] \label{lem:EfficientDistinguishersAreTight}
    Let~$\{A, B\}$ be a separation of a locally finite graph~$G$ which efficiently distinguishes a pre-tangle~$P$ and a tangle~$Q$ in~$G$.
    Suppose that~$(A, B) \in Q$.
    Then there exists a tight component~$K$ of~$G - (A \cap B)$ with~$V( K ) \subseteq B \setminus A$.
    In particular, if~$P$ is also a tangle, then~$\{A, B\}$ is tight. \looseness=-1
\end{lemma}

\begin{proof}[Proof of \cref{main:Technical}]
    Each separation in the tree of tangles~$N$ efficiently distinguishes some two tangles in~$G$ and is hence tight by~\cref{lem:EfficientDistinguishersAreTight}.
    Since every tangle in~$G$ is also a pre-tangle in~$G$, \cref{thm:ThickEnd} implies that the unique end 
    in the closure of~$A \cap B$ is thick, as desired.
\end{proof}

\begin{proof}[Proof of \cref{main:Theorem}]
    By~\cref{thm:ToT}, there exists a tree of tangles~$N$ of~$G$.
    Since~$G$ does not have any thick end, \cref{main:Technical} implies that every strictly increasing sequence in~$\vN$ is exhaustive, so~$N$ is itself exhaustive.
    Therefore, $N$ induces a canonical tree-decomposition of~$G$ by~\cref{lem:ExhaustiveImpliesTreeDecomp}, as all the separations in~$N$ are tight and hence proper.
    This tree-decomposition of~$G$ then efficiently distinguishes all distinguishable tangles in~$G$ by construction and hence is as desired.
\end{proof}

If one considers, instead of tangles, the more general `profiles' as introduced in~\cite{ProfilesNew}, then there is a direct generalisation of~\cref{thm:ToT}:
Elbracht, Kneip and Teegen showed in~\cite{InfiniteSplinters}*{Theorem 6.6} that for a robust set~$\cP$ of regular profiles in a connected locally finite graph, there is a canonical nested of~$\cP$-relevant separations which efficiently distinguishes~$\cP$.
Since profiles are also pre-tangles and every tangle, in particular the ones induced by ends, are profiles, one can then obtain direct analogues of~\cref{main:Theorem} and~\cref{main:Technical} for sets~$\cP$ of profiles with the same proofs.

\section{Infinitely many rays beyond the limit} \label{app:InfRaysBeyondLimit}

In this section we prove a structurally refined version of~\cref{thm:ThickEnd}:
while~\cref{thm:ThickEnd} asserts that the (unique) end~$\omega$ of~$G$ 
in the closure of the separator of the limit separation~$(A, B)$ is thick, we here show in~\cref{prop:ThicknessBeyondLimit} that there are infinitely many disjoint rays in~$G[B \setminus A]$ which are contained in the end~$\omega$; in other words, these rays witness the thickness of~$\omega$ `beyond' the limit separation~$(A, B)$.

\begin{proposition} \label{prop:ThicknessBeyondLimit}
    Let~$\cP$ be a set of pre-tangles in a connected locally finite graph~$G$, and let~$N$ be a nested set of~$\cP$-relevant tight separations of~$G$ which efficiently distinguishes~$\cP$.
    Let~$((A_i, B_i))_{i \in \N}$ be a strictly increasing sequence in~$\vN$, and let~$(A, B)$ be its limit.

    If~$((A_i, B_i))_{i \in \N}$ is non-exhaustive, then the unique end~$\omega$ of~$G$ which is 
    in the closure of~$A \cap B$ contains an infinite family of disjoint rays in~$G[B \setminus A]$.
\end{proposition}

As preparation for the proof, we remark that following the proof of the ubiquity of rays~\cite{halin65}*{Satz~1} by Halin one easily obtains the following (see also \cite{DiestelBook16}*{Theorem~8.2.5 \& Exercise~43}).
\begin{theorem}\label{ubiquityofrays}
    An end $\omega$ in a graph $G$ contains infinitely many disjoint rays, if it contains $k$ many disjoint rays for every $k \in \N$. \qed
\end{theorem}

\begin{proof}[Proof of \cref{prop:ThicknessBeyondLimit}]
    We first proceed as in the proof of~\cref{thm:ThickEnd} to obtain a strictly increasing sequence~$((A_i', B_i'))_{i \in \N}$ in~$\vN$ with the same limit~$(A, B)$ and a corresponding sequence~$(P_i)_{i \in \N}$ in~$\cP$ satisfying \cref{item:Orientation} and~\cref{item:DistinguishingSequence}.
    Moreover, the~$(A_i', B_i')$ have strictly increasing order and, by the claim, every~$(A_i', B_i')$ efficiently distinguishes~$P_i$ and~$P_\omega$, the pre-tangle in~$G$ induced by~$\omega$.

    We claim that it suffices that there are~$k$~disjoint rays in~$G[B \setminus A]$ which belong to~$\omega$ in~$G$ for every~$k \in \N$.
    Indeed, let $Y$ be the set of all ends of $G[B \setminus A]$ which contain a ray in $\omega$.
    If $Y$ is infinite, then one greedily finds the desired infinitely many disjoint rays in $G[B\setminus A]$ in $\omega$.
    Otherwise $Y$ is finite.
    Thus, some end $\omega' \in Y$ contains $k$ disjoint rays for every $Y$ by the pigeonhole principle.
    Hence by \cref{ubiquityofrays}, $\omega'$ contains infinitely many disjoint rays.
    These are also contained in $\omega$ by the definition of $Y$.

    So consider some fixed~$k \in \N$.
    Let~$X:= A \cap B$, and fix any set~$Z \subseteq X$ of size~$k$.
    By~\cref{lem:PushingLemma}, there exists~$I \in \N$ such that~$Z \subseteq X_i := A'_i \cap B'_i$ for all~$i \ge I$.
    Since every~$\{A_i', B_i'\}$ efficiently distinguishes~$P_i$ from~$P_\omega$, no~$X_i$--$\omega$ separator of~$G$ of size less than~$|X_i|$ exists.
    Thus, it follows from~\cite{eme}*{Theorem~1.1} that there is a family~$\cR_i$ of~$|X_i|$ many disjoint~$X_i$--$\omega$~rays in~$G$ for every~$i \in \N$.
    We remark that every vertex in~$X_i$ is the starting vertex of some ray in~$\cR_i$.
    For every~$i \ge I$, we fix the subfamily~$\cR'_i$ of~$\cR_i$ containing precisely those rays starting in~$Z$.

    For every~$\ell \in \N$, let~$\cM^\ell$ be the set consisting of those families of~$k$ paths of length~$\ell$ which appear jointly as initial segments of the rays in~$\cR'_i$ for infinitely many~$i \geq I$; note that the paths in such families are disjoint and start in~$Z$.
    Since the graph~$G$ is locally finite, there are only finitely many paths of length~$\ell$ in~$G$ which start in~$Z$.
    To see that~$\cM^\ell$ is non-empty, we can thus apply the pigeonhole principle to the infinite collection of families of~$k$ paths of length~$\ell$ starting in~$Z$ which are given by the initial segments of the rays in the~$\cR_i'$.
    
    Given any family~$\{R_1, \dots, R_k\} \in \cM^\ell$, we have~$\{R_1', \dots, R_k'\} \in \cM^{\ell-1}$ where~$R_i'$ is the subpath of~$R_i$ of length~$\ell-1$ which still starts in~$Z$.
    Thus, König's infinity lemma~\cite{InfLemma}*{VI. Kapitel~\S\,2,~Satz~6} (see also \cite{DiestelBook16}*{Lemma~8.1.2})
    applied to~$(\cM^\ell)_{\ell \in \N}$ yields for every~$\ell \in \N$ a family~$\{R^\ell_1, \dots, R^\ell_k\} \in \cM^\ell$ such that~$R^{\ell-1}_i$ is the subpath of~$R^\ell_i$ of length~$\ell-1$ starting in~$Z$ for~$i = 1, \dots, k$.
    Let~$R_i$ be the ray~$\bigcup_{\ell \in \N} R^\ell_i$ for~$i = 1, \dots, k$.
    These rays~$R_i$ are disjoint as the~$R^\ell_i$ are disjoint for every~$\ell \in \N$ by the definition of~$\cM^\ell$.
    We claim that~$R_1, \dots, R_k$ are as desired, i.e.\ a family of~$k$ disjoint rays in~$\omega$ which are contained in~$G[B \setminus A]$ (except their startvertices).
    
    First, $R_1, \dots, R_k$ are contained in~$G[B \setminus A]$ except their startvertices which are in~$Z \subseteq A \cap B$.
    Indeed, let~$v$ be a vertex of some~$R_i$ which is not its startvertex.
    By the definition of~$R_i$ and~$(\cM^\ell)_{\ell \in \N}$, $v$ is contained in a ray in~$\cR_i$ for infinitely many~$i \geq I$. 
    All the rays in~$\cR_i$ avoid~$X_i$ except in their startvertices.
    Hence, every ray in~$\cR_i$ is contained in~$G[B'_i \setminus A'_i]$ except its startvertex, since~$\omega$ is 
    in the closure of~$A \cap B \subseteq B'_i$. 
    Altogether, $v$ is contained in~$B_i \setminus A_i$ for infinitely many~$i \geq I$ and thus in~$B \setminus A$ by \cref{lem:PushingLemma}.

    It remains to show that~$R_1, \dots, R_k$ belong to the end~$\omega$ of~$G$.
    Consider any~$R := R_i$, and let~$Q$ be an arbitrary ray in~$\omega$.
    We now inductively exhibit infinitely many disjoint paths in~$G$ joining~$R$ and~$Q$ which then yields that they are in the same end of~$G$ and, in particular, that $R \in \omega$.
    Every vertex in~$R \cap Q$ is a trivial path joining~$R$ and~$Q$.
    So if~$R \cap Q$ is infinite, then we are done.
    Suppose that~$R \cap Q$ is finite.

    Let us assume that we have already found disjoint paths~$P_1, \dots, P_j$ joining~$R$ and~$Q$ for some~$j \in \N$.
    We seek an additional path~$P_{j+1}$ joining~$R$ and~$Q$ which is disjoint from all the previous~$P_i$.
    Since~$P_1, \dots, P_j$ are finite paths, there are finite initial segments~$R^*$ and~$Q^*$ of~$R$ and~$Q$, respectively, such that~$R - R^*$ and~$Q - Q^*$ are both disjoint from~$P_1, \dots, P_j$.
    By the construction of the~$R_i$, the initial segment of~$R$ which consists of~$R^*$ together with the following edge is also the initial segment of some ray~$Q_{j+1} \in \omega$.
    Since both $Q$ and $Q_{j+1}$ are in $\omega$, there are infinitely many disjoint~$Q$--$Q_{j+1}$~paths, and we especially find one such path~$P_{j+1}'$ which is disjoint from both~$R^*$ and~$Q^*$ as well as the paths~$P_1, \dots, P_j$.
    We then let~$P_{j+1}$ be the (unique) path 
    which first follows~$Q_{j+1} - R^*$ and then~$P_{j+1}'$.
    This path~$P_{j+1}$ starts in~$R$ and ends in~$Q$ by the choice of~$Q_{j+1}$ and~$P_{j+1}'$, and it is disjoint from~$P_1, \dots, P_j$ by construction, as desired.
    Iterating this recursive construction yields the desired set of infinitely many disjoint~paths starting in~$R$ and ending in~$Q$.
\end{proof}

A natural strengthening of~\cref{prop:ThicknessBeyondLimit} would assert that the infinitely many rays in~$G[B \setminus A]$ not only belong to the same end~$\omega$ of~$G$, but even to the same end of~$G[B \setminus A]$.
This, however, does not hold as the following example demonstrates; this example is an adaption of~\cite{InfiniteSplinters}*{Example~4.9}.

\begin{example} \label{ex:RaysNotEquivalent}
    We construct a graph~$G$ (see \cref{fig:example}) as follows.

    \begin{figure}[ht]
        \centering
        \includegraphics{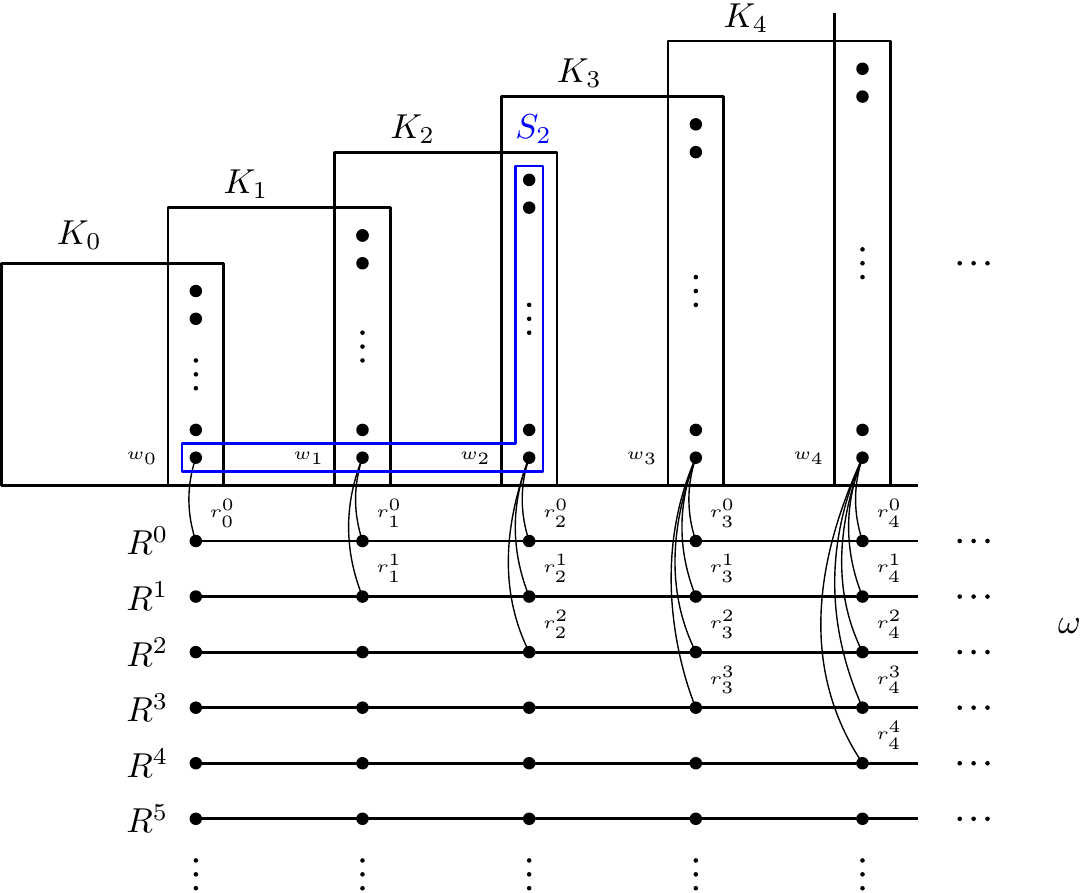}
        \caption{The graph $G$ from \cref{ex:RaysNotEquivalent}}
        \label{fig:example}
    \end{figure}
    
    For every~$n \in \N$, we pick a copy $K_n$ of the complete graph~$K^{2^{n+4}}$
    on~$2^{n+4}$
    vertices.
    In this~$K^{2^{n+4}}$, 
    we fix~$2^n$ vertices~$u_1^n, \dots, u_{2^n}^n$ and~$2^{n+1}$ vertices~$v_1^n, \dots, v_{2^{n+1}}^n$ all of which are distinct.
    We then identify~$u_i^{n+1}$ with~$v_i^n$ for all~$n \in \N$ and $i \leq 2^{n+1}$.
    Furthermore, we pick a ray~$R^n = r_0^n r_1^n r_2^n \dots$ for every~$n \in \N$ such that these rays are disjoint from each other and from the graph constructed before.
    Then for every~$m \in \N$ and~$n \le m$, we add an edge joining~$r_m^n$ and~$w^m := u_1^{m+1} = v_1^m$.
    This completes the construction of the graph~$G$.

    Each of the~$K_n$
    induces a $2^{n+2}$-tangle~$P_n$ in~$G$.
    It is easy to check that the tight separation~$s_n$ with separator~$S_n := \{v_1^i \mid i < n\} \cup \{v_j^n \mid j \le 2^{n+1}\}$ efficiently distinguishes~$P_n$ and~$P_{n+1}$.
    Moreover, the nested set~$N = \{s_n \mid n \in \N\}$ efficiently distinguishes the set~$\cP = \{P_n \mid n \in \N\}$ of (pre-)tangles in~$G$.
    
    Now let~$\vs_n$ be the orientation of~$s_n$ with~$\vs_n \in P_{n+1}$.
    Then~$(\vs_n)_{n \in \N}$ is a strictly increasing sequence in~$\vN$, and we write~$\vs = (A, B)$ for its limit separation.
    We have~$B \setminus A = \bigcup_{n \in \N} V(R^n)$.
    Hence, the sequence~$(\vs_n)_{n \in \N}$ is non-exhaustive, and~$G[B \setminus A]$ is the disjoint union of the rays~$R^n$.
    In particular, $G[B \setminus A]$ does not contain an infinite family of equivalent rays.
\end{example}

The above~\cref{ex:RaysNotEquivalent} can be adapted to show that there need not even exist infinitely many rays in~$G[B \setminus A]$ which are equivalent in~$G[B]$.
Instead of joining the rays $R_n$ as described above, we join each $R_n$ to the set $A \cap B = \{v^0_1, v^1_1, v^2_1, \dots\}$ cofinally in such a way that each $v^i_1$ is joined to precisely $i$ of the rays and every two rays~$R^n$ and~$R^m$ share only finitely many neighbours in~$A \cap B$. 
Now $s_n$ chosen as above is again a tight separation which efficiently distinguishes the tangles $P_n$ and $P_{n+1}$ induced by the respective copies of $K^{2^{n+4}}$ and $K^{2^{(n+1)+4}}$.
Now every two disjoint rays in $G[B \setminus A]$ are not equivalent, as desired.
We omit the detailed construction here.

\section*{Acknowledgements}

\noindent We thank Jan Kurkofka for an insightful discussion concerning the proof of~\cref{lem:LimitSepHasInfiniteOrder} and Nathan Bowler for his idea how to prove~\cref{prop:ThicknessBeyondLimit}. 
We also thank the reviewers for their valuable comments that greatly improved the presentation of some arguments and fixed a few mistakes. Especially, we are grateful for Reviewer 1's shorter and more intuitive proof of \cref{lem:FiniteOrderCrossLimit} and Reviewer 2's input regarding the presentation of \cref{app:InfRaysBeyondLimit}.

The first author gratefully acknowledges support by doctoral scholarships of the Studienstiftung des deutschen Volkes and the Cusanuswerk -- Bisch\"{o}fliche Studienf\"{o}rderung.
The second author gratefully acknowledges support by a doctoral scholarship of the Studienstiftung des deutschen Volkes.

\bibliographystyle{amsplain}
\bibliography{collective_adapted}

\end{document}

%% file: Preamble.tex
\usepackage{amsmath}
\usepackage{amssymb}
\usepackage{amsthm}
\usepackage{mathtools}
\usepackage{letterswitharrows}
\usepackage{enumitem}
\setenumerate{label={\normalfont (\roman*)}}

\usepackage[utf8]{inputenc}
\usepackage[T1]{fontenc}
\usepackage{lmodern}
\usepackage[babel]{microtype}
\usepackage[english]{babel}
\usepackage{relsize}

\usepackage{graphicx}
\usepackage{subcaption}

\usepackage{geometry}
\usepackage{xcolor} 	
\usepackage{hyperref}
\hypersetup{
	colorlinks,
	linkcolor={red!60!black},
	citecolor={green!60!black},
	urlcolor={blue!60!black},
}
\usepackage[abbrev, msc-links]{amsrefs}
\usepackage[nameinlink, capitalise, noabbrev]{cleveref}
\crefformat{enumi}{#2#1#3}
\crefformat{equation}{#2(#1)#3}
\crefname{mainresult}{Theorem}{Theorems}
\usepackage{doi}

\renewcommand{\PrintDOI}[1]{\doi{#1}}

\let\setminus=\smallsetminus

\renewcommand{\leq}{\leqslant}
\renewcommand{\geq}{\geqslant}
\renewcommand{\ge}{\geq}
\renewcommand{\le}{\leq}

\newtheorem{theorem}{Theorem}[section] 
\newtheorem{proposition}[theorem]{Proposition}

\newtheorem{lemma}[theorem]{Lemma}

\newtheorem{mainresult}{Theorem} 

\theoremstyle{definition}
\newtheorem{example}[theorem]{Example}

\theoremstyle{remark}

\newtheorem*{claim*}{Claim}

\crefname{claim}{Claim}{Claims}

\newenvironment{claimproof}{\noindent\textit{Proof of the claim.}}{\hfill\ensuremath{\blacksquare}\medskip}
\usepackage{etoolbox}

\newcommand{\COMMENT}[1]{{}}

\let\theta=\vartheta
\let\rho=\varrho
\let\phi=\varphi
\def\N{\mathbb N}

\makeatletter

\def\calCommandfactory#1{%
  \expandafter\def\csname c#1\endcsname{\mathcal{#1}}}
\def\frakCommandfactory#1{%
  \expandafter\def\csname frak#1\endcsname{\mathfrak{#1}}}
   
\newcounter{ctr}
\loop
  \stepcounter{ctr}
  \edef\X{\@Alph\c@ctr}
  \expandafter\calCommandfactory\X
  \expandafter\frakCommandfactory\X
\ifnum\thectr<26
\repeat




%% file: main.bbl
\begin{bibdiv}
\begin{biblist}

\bib{eme}{article}{
      author={Bruhn, H.},
      author={Diestel, R.},
      author={Stein, M.},
       title={{M}enger's theorem for infinite graphs with ends},
        date={2005},
     journal={J.~Graph Theory},
      volume={50},
       pages={199\ndash 211},
}

\bib{confing}{article}{
      author={Carmesin, J.},
      author={Diestel, R.},
      author={Hundertmark, F.},
      author={Stein, M.},
       title={Connectivity and tree structure in finite graphs},
        date={2014},
     journal={Combinatorica},
      volume={34},
      number={1},
       pages={1\ndash 35},
      eprint={1105.1611},
}

\bib{CanonicalTreesofTDs}{article}{
      author={Carmesin, Johannes},
      author={Hamann, Matthias},
      author={Miraftab, Babak},
       title={Canonical trees of tree-decompositions},
        date={2022},
     journal={Journal of Combinatorial Theory, Series B},
      volume={152},
       pages={1\ndash 26},
      eprint={2002.12030},
}

\bib{EndsAndTangles}{article}{
      author={Diestel, R.},
       title={Ends and tangles},
        date={2017},
     journal={Abhandlungen Math.\ Sem.\ Univ.\ Hamburg},
      volume={87{\rm , Special issue in memory of Rudolf Halin}},
       pages={223\ndash 244},
      eprint={1510.04050},
}

\bib{DiestelBook16}{book}{
      author={Diestel, R.},
       title={Graph theory \emph{(5th edition)}},
   publisher={Springer-Verlag},
        date={2017},
        note={Electronic edition available at
  \href{http://diestel-graph-theory.com/}{diestel-graph-theory.com}},
}

\bib{ASS}{article}{
      author={Diestel, R.},
       title={Abstract separation systems},
        date={2018},
     journal={Order},
      volume={35},
       pages={157\ndash 170},
      eprint={1406.3797},
}

\bib{TreeSets}{article}{
      author={Diestel, R.},
       title={Tree sets},
        date={2018},
     journal={Order},
      volume={35},
       pages={171\ndash 192},
      eprint={1512.03781},
}

\bib{ProfilesNew}{article}{
      author={Diestel, R.},
      author={Hundertmark, F.},
      author={Lemanczyk, S.},
       title={Profiles of separations: in graphs, matroids, and beyond},
        date={2019},
     journal={Combinatorica},
      volume={39},
      number={1},
       pages={37\ndash 75},
      eprint={1110.6207},
}

\bib{GraphDec}{article}{
      author={Diestel, R.},
      author={Jacobs, R.W.},
      author={Knappe, P.},
      author={Kurkofka, J.},
       title={Canonical graph decompositions via coverings},
        date={2022},
      eprint={2207.04855},
}

\bib{InfiniteSplinters}{article}{
      author={Elbracht, C.},
      author={Kneip, J.},
      author={Teegen, M.},
       title={Trees of tangles in infinite separation systems},
        date={2021},
     journal={Math. Proc. Camb. Phil. Soc.},
       pages={1\ndash 31},
      eprint={2005.12122},
}

\bib{TreelikeSpaces}{article}{
      author={Gollin, P.},
      author={Kneip, J.},
       title={Representations of infinite tree sets},
        date={2020},
     journal={Order},
      eprint={1908.10327},
}

\bib{gollin2022characterising}{article}{
      author={Gollin, Pascal},
      author={Heuer, Karl},
       title={Characterising k-connected sets in infinite graphs},
        date={2022},
     journal={Journal of Combinatorial Theory, Series B},
      volume={157},
       pages={451\ndash 499},
      eprint={1811.06411},
}

\bib{halin65}{article}{
      author={Halin, R.},
       title={{\"U}ber die {M}aximalzahl fremder unendlicher {W}ege in
  {G}raphen},
        date={1965},
     journal={Math.\ Nachr.},
      volume={30},
       pages={63\ndash 85},
}

\bib{InfLemma}{book}{
      author={K\"onig, D.},
       title={Theorie der endlichen und unendlichen {G}raphen},
   publisher={Akademische Verlagsgesellschaft},
        date={1936},
}

\bib{GM}{article}{
      author={Robertson, N.},
      author={Seymour, P.D.},
       title={Graph minors {I--XX}},
        date={1983--2004},
     journal={Journal of Combinatorial Theory, Series~B},
}

\bib{GMX}{article}{
      author={Robertson, N.},
      author={Seymour, P.D.},
       title={Graph minors. {X}. {O}bstructions to tree-decomposition},
        date={1991},
     journal={Journal of Combinatorial Theory, Series~B},
      volume={52},
       pages={153\ndash 190},
}

\end{biblist}
\end{bibdiv}
